\documentclass[twoside]{article}

\usepackage[accepted]{aistats2025}

\usepackage[dvipsnames]{xcolor}
\usepackage{algorithm}
\usepackage{algorithmic}
\usepackage{subcaption}
\usepackage{graphicx}
\usepackage{mathtools}
\usepackage{amssymb}
\usepackage{amsmath}
\usepackage{amsfonts}
\usepackage{hyperref}
\usepackage{bm}
\usepackage{todonotes}
\usepackage{caption}
\usepackage{subcaption}
\usepackage{tabularx}
\usepackage{booktabs}

\usepackage{amsthm}
\theoremstyle{plain}
\newtheorem{theorem}{Theorem}[section]

\newtheorem{lemma}[theorem]{Lemma}

\theoremstyle{definition}

\theoremstyle{remark}
\newtheorem{remark}[theorem]{Remark}

\usepackage[round]{natbib}

\usepackage{xcolor}

\begin{document}

%
\runningtitle{PK-MIQP}

%

\twocolumn[

\aistatstitle{Global Optimization of Gaussian Process Acquisition Functions Using a Piecewise-Linear Kernel Approximation}

\aistatsauthor{Yilin Xie \And Shiqiang Zhang \And Joel A. Paulson \And Calvin Tsay}

\aistatsaddress{Imperial College London \And  Imperial College London \And The Ohio State University \And Imperial College London} 
]

\begin{abstract}
    Bayesian optimization relies on iteratively constructing and optimizing an acquisition function. The latter turns out to be a challenging, non-convex optimization problem itself. Despite the relative importance of this step, most algorithms employ sampling- or gradient-based methods, which do not provably converge to global optima. 
    This work investigates mixed-integer programming (MIP) as a paradigm for \textit{global} acquisition function optimization. Specifically, our Piecewise-linear Kernel Mixed Integer Quadratic Programming (PK-MIQP) formulation introduces a piecewise-linear approximation for Gaussian process kernels and admits a corresponding MIQP representation for acquisition functions. The proposed method is applicable to uncertainty-based acquisition functions for any stationary or dot-product kernel. We analyze the theoretical regret bounds of the proposed approximation, and empirically demonstrate the framework on synthetic functions, constrained benchmarks, and a hyperparameter tuning task.

\end{abstract}

\section{INTRODUCTION}\label{sec:introduction}

Optimization of black-box functions is a compelling task in many scientific fields, such as hyperparameter tuning for machine learning~\citep{Ranjit2019DLBO}, routing for robot control problems~\citep{Nambiar2022BOrobot}, designing energy systems~\citep{thebelt2022multi}, and drug discovery~\citep{Colliandre2024BOdrug}. These tasks share challenges such as limited knowledge about the true underlying objective function and expensive evaluations on the target function, i.e., experiments. Given the latter, finding a high-quality sample point at each evaluation becomes extremely valuable. 

Bayesian optimization (BO) is a popular class of algorithms designed for this setting. BO can be divided into two main components: (1) a Bayesian model of the objective function (usually a Gaussian process (GP)), and (2) an acquisition function to decide which point $\bm{x^*}$ to sample next. An acquisition function is a mathematical expression that quantifies the value of evaluating a particular point in the search space. 
The choice of acquisition function balances exploration and exploitation~\citep{paulson2024bayesian}. 
Mathematical optimization of the acquisition function is thus a key step of BO. While many works formulate more indicative and sophisticated acquisition functions, comparatively less attention has been given to computing the optimal solutions of acquisition functions. While a well-formulated acquisition function is critical, its reliable optimization can be just as beneficial for BO performance~\citep{Wilson2018MaxAcq,kim2021local}.

\begin{figure*}[t]
  \includegraphics[width=0.49\textwidth]{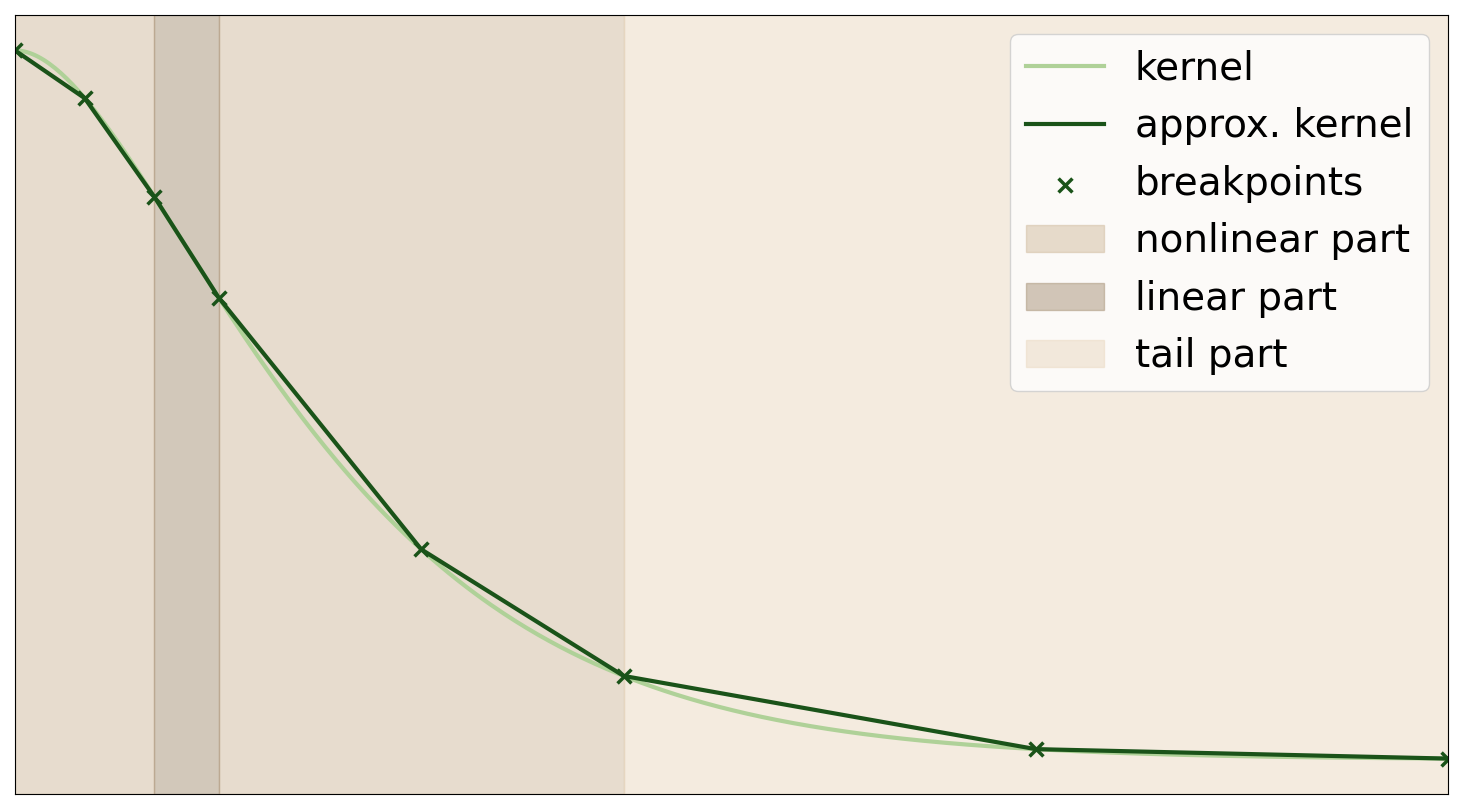}
   \includegraphics[width=0.49\textwidth]{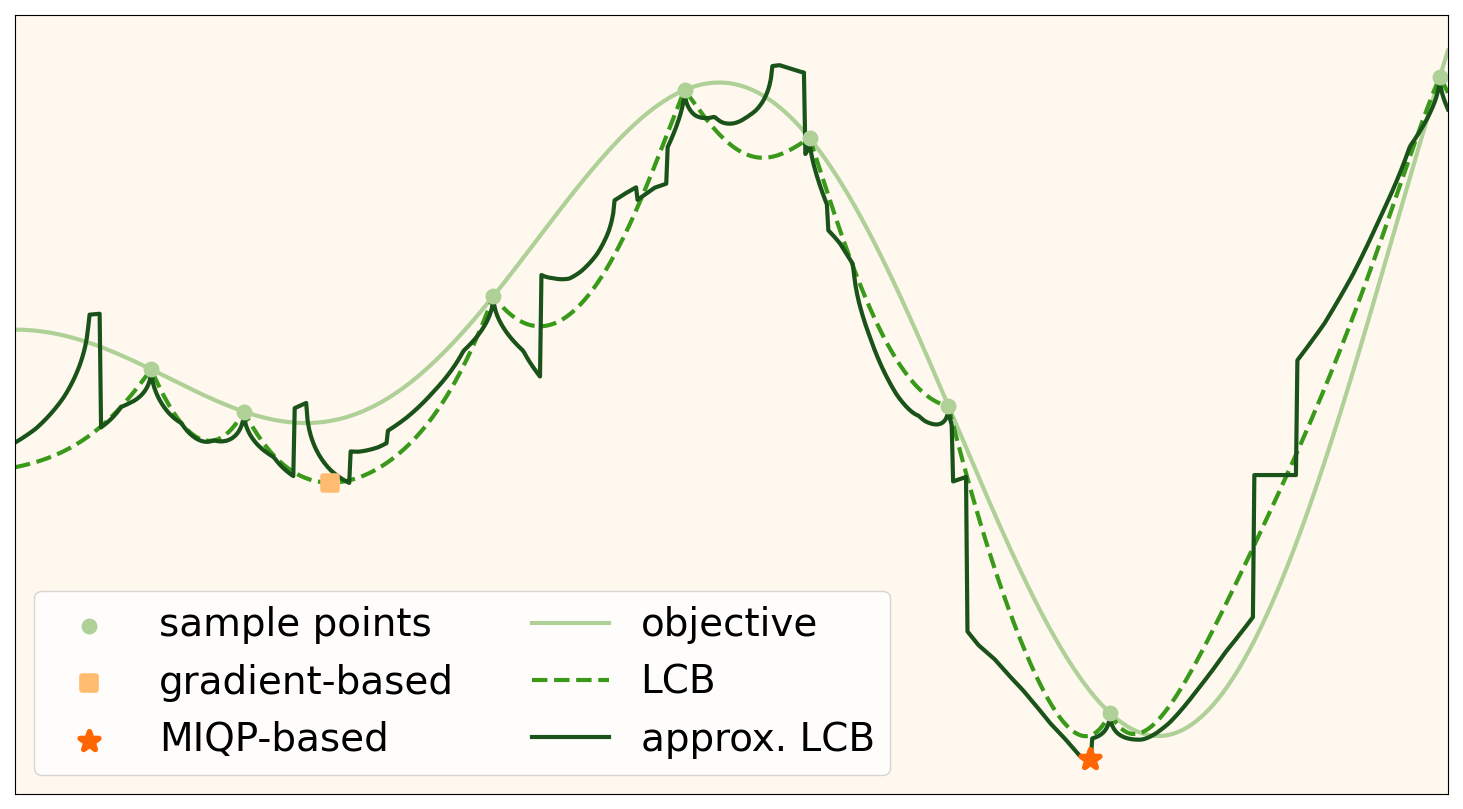}
  \caption{(\textbf{left}) Illustration of piecewise linear approximation of kernel function. (\textbf{right}) Visualization of the effect of kernel approximation on LCB acquisition function. The solution from gradient-based method (orange square) may end up at a local minimum, a sampling-based solution can miss the global minimum, and optimizing approximated LCB using global model (red star) will provide the global solution.}
  \label{idea_illus}
\end{figure*}

Gradient- and sampling-based methods remain the mainstream for acquisition function optimization. While the objective function of BO is a black box, closed-form expressions are available for many acquisition functions, e.g., Upper Confidence Bound (UCB) \citep{Srinivas2010GP}, Expected Improvement (EI) \citep{Jones1998EfficientGO}. 
Therefore, mathematical programming, especially (stochastic) gradient-based methods such as L-BFGS-B \citep{Zhu1997LBFGS} can perform well given first- and sometimes second-order derivatives. However, derivative information can be unreliable, e.g., due to numerical issues \citep{Ament2024unexpected}. Gradient-based methods can also be trapped at local optima and thus return sub-optimal solutions. 
On the other hand, \textit{sampling} methods 
are invariably limited by the curse of dimensionality and are also likely to return sub-optimal solutions, given limited evaluation budgets. 
Algorithms that provide a \textit{global} optimality guarantee for solving acquisition functions are thus lacking.

An alternative class of \textit{deterministic} methods is mixed-integer programming (MIP)~\citep{belotti2013mixed}, where established algorithms can globally optimize an objective function subject to various constraints. Conditions of a given problem are expressed as linear, quadratic, nonlinear and/or integral constraints and are passed to modern solvers that utilize branch-and-bound (B\&B) algorithms. 
By bounding the objective value, B\&B algorithms provide some guarantees on the global optimality of the solution. To the best of our knowledge, the application of MIP in BO has only been studied by \citet{Schweidtmann2021GPMIP}. This work further explores MIP-based global optimization for BO, specifically with novel kernel approximation techniques tailored to the solution methods.

In this paper, we propose Piecewise-linear Kernel Mixed Integer Quadratic Programming (PK-MIQP), a global optimization framework for GP acquisition functions. Specifically, we introduce a piecewise linear approximation of the GP kernel function (Figure \ref{idea_illus}, left) that enables a mixed-integer quadratic programming (MIQP) formulation of the acquisition function. 
Note that the model becomes mixed-integer linear if only the GP mean is required. 
Lower confidence bound (LCB) is chosen as a representative example to demonstrate the proposed framework. We then use a B\&B solver to globally optimize the approximated acquisition function. While gradient-based methods can return sub-optimal solutions without further indication (Figure \ref{idea_illus}, right), our method approximates the global optimum within a bounded neighborhood. The main contributions of this work are as follows:

\begin{enumerate}
    \item We propose a MIP-motivated piecewise linear approximation for GP kernel functions in stationary or dot-product form.
    \item We present an MIQP formulation to solve the resulting acquisition function optimization problems with global optimality guarantees.
    \item We theoretically analyze the performance and worst-case error of PK-MIQP. 
    \item We embed PK-MIQP in a full BO procedure and evaluate its performance against state-of-the-art optimization methods on tasks including synthetic functions, constrained benchmarks, and a hyperparameter tuning task.
\end{enumerate}

\section{RELATED WORK}\label{sec:related_work}
Recent research in the field of BO often focuses on solving long-standing problems, such as scaling BO to perform on problems with higher dimensionalities \citep{Spagnol2019Varselect, Cartis2021RanEmbed, Eriksson2021saas} and improving computational efficiency through more informative acquisition functions \citep{Oh2019bock, Ath2021greed, Ament2024unexpected}. While these innovations can be important and effective, the acquisition functions in these works remain optimized using classic and traditional methods. 

Gradient-based methods such as L-BFGS-B \citep{Zhu1997LBFGS} and stochastic gradient ascent \citep{Kingma2017adam} are popular choices to optimize acquisition functions in BO. For example, \citet{Eriksson2021saas} optimize the EI acquisition function based on a sparse axis-aligned subspace GP using L-BFGS-B \citep{Zhu1997LBFGS}. \citet{Oh2019bock} propose a cylindrical transformation of the search space to allow the chosen acquisition function to explore more near the center of search space using the Adam algorithm \citep{Kingma2017adam}. Though widely-used, solving acquisition functions using gradient-based methods has its limitations. \citet{Ament2024unexpected} highlight the issue of vanishing values and gradients when applying gradient-based methods to acquisition function optimization. Specifically, improvement-based acquisition functions (e.g., EI) can suffer from numerically zero acquisition values and gradients, making gradient-based methods return sub-optimal solutions due to lack of gradient information. \citet{Ament2024unexpected} thus propose a reformulation of the EI acquisition family, termed LogEI, where functions retain the same optima, but more stable in terms of gradient values; the authors demonstrated performance in BO using L-BFGS-B to solve the LogEI acquisition functions. \citet{Daulton2022ProbReparam} also address the limitations of gradient-based methods on maximizing acquisition functions in discrete and mixed search space. The authors propose a novel probabilistic representation of the acquisition function, where continuity of parameters in the acquisition function is restored, enabling the use of the gradient-based method Adam. Popular BO tools, e.g., BoTorch \citep{balandat2020botorch}, often consider multi-start gradient approaches to overcome local optima. However, these approaches may have difficulty in determining the non-flat regions of acquisition functions, especially in higher dimensional settings \citep{rana2017egp, adebiyi2024tsroots}.

As an alternative direction, sampling methods such as Nelder-Mead \citep{Nelder1965neldermead} are also a popular tool for acquisition function optimization in BO methods. While gradient-based methods suffer from difficulties in getting reliable gradient information, sampling methods mainly suffer from the curse of dimensionality. Here, \citet{Kandasamy2015AGP} propose an additive structure to model the objective function to exploit the efficiency of sampling methods in low-dimensional problems. \citet{Eriksson2020scalable} employ Thompson sampling in local trust regions and search for global optima through multiple local search. 

Global optimality in BO has primarily been discussed in terms of the black-box objective function, rather than the acquisition function evaluated at each iteration. Towards the latter, some researchers heuristically search for global optimality through multi-start local methods \citep{Eriksson2020scalable, Mathesen2021AdapRestart} or evolutionary algorithms such as Firefly \citep{song2024vizier}, CMA-ES \citep{hansen2006cma, Wilson2018MaxAcq} and TSEMO \citep{Bradford2018tsemo}. MIP comprises a well-known paradigm to solve optimization problems globally, but to-date has been rarely used in BO. Some works apply MIP to solve acquisition functions based on more compatible surrogate models, e.g., neural networks \citep{Papalexopoulos2022discreteMIP} or decision trees \citep{thebelt2022multi}. MIP is also used to handle discrete decisions in hybrid/combinatorial BO \citep{baptista2018bocs, daxberger2020mixBO, deshwal2021mercer}. Relatively few works apply MIP with GP-based surrogate models in BO. \citet{thebelt2022tree} formulate GPs using tree kernels and solve the resulting acquisition function as an MIQP. \citet{Schweidtmann2021GPMIP} formulate acquisition functions for smooth-kernel GPs as a mixed-integer nonlinear program (MINLP), which is then solved using a B\&B algorithm. However, MINLP can easily exceed computational budgets in real applications. 
This work seeks a more stable and practical algorithm for standard GPs that combines some global optimality guarantees, with more realistic computing times. 


\section{BACKGROUND}\label{sec:background}


\subsection{Gaussian Processes}\label{subsec:GP}
A Gaussian process (GP) models a joint multivariate Gaussian distribution over some random variables \citep{Schulz2018GPtutorial}. A GP is fully specified by a prior mean function $\mu(\cdot)$ and a kernel $K(\cdot,\cdot)$:
\begin{equation*}
    f(\cdot) \sim \mathcal{GP}(\mu(\cdot), K(\cdot, \cdot))
\end{equation*}
Normally the prior mean is set to zero to simplify the subsequent posterior computation, leaving only the kernel function. Common choices here are the squared exponential (SE) kernel or $Mat\acute{e}rn$ kernel. 


GP regression can be viewed as a Bayesian statistical approach for modelling and predicting functions. Given $t$ observed data points $\{(\bm{x_i},f(x_i))\}_{i=1}^t$, we denote $\bm{X}=[\bm{x_1},\dots,\bm{x_t}]$, $\bm{y}=[f(\bm{x_1}),\dots,f(\bm{x_t})]$.
For a new point $\bm{x} \in \mathbb{R}^n$, the posterior mean and variance of its function value is given by \citet{Frazier2018BOtutorial}:
\begin{equation}\label{eq:GP_mean_variance}
    \begin{aligned}
        \mu_t(\bm{x})&=K_{xX}K_{XX}^{-1}\bm{y} \\
        \sigma_t^2({\bm x})&= K_{xx}-K_{xX}K_{XX}^{-1}K_{Xx}
    \end{aligned}
\end{equation}
where we omit the conditioning on $(\bm{X}, \bm{y})$ and use $K_{XY}=K(\bm{X},\bm{Y})$ for simplicity. Note that $K_{xx}$ is the kernel variance $\sigma_f^2$.

When the observed output is noisy, i.e., $y=f(\bm{x})+\epsilon$, we assume $\epsilon\sim\mathcal N(0,\sigma_\epsilon^2)$, giving a posterior mean and variance of:
\begin{equation}\label{eq:GP_mean_variance_with_noise}
    \begin{aligned}
        \mu_t(\bm{x})&=K_{xX}(K_{XX}+\sigma_\epsilon^2\bm{I})^{-1}\bm{y} \\
        \sigma_t^2({\bm x})&= K_{xx}-K_{xX}(K_{XX}+\sigma_\epsilon^2\bm{I})^{-1}K_{Xx}
    \end{aligned}
\end{equation}
Nevertheless, for simplicity, we omit the noise term and assume that $\|K_{XX}\|_2\ge \sigma^2_\epsilon$ in our proofs.

\subsection{Bayesian Optimization}\label{subsec:BO}


A general BO procedure constructs an acquisition function (AF) based on the above GP and seeks to maximize its value. 
We consider the classic lower confidence bound (LCB) acquisition function:
\begin{equation*}
    \alpha_\mathit{LCB}(\bm{x}) = \mu_t(\bm{x}) - \beta_t^{1/2} \sigma_t(\bm{x})
\end{equation*}
Here $\mu_t$ and $\sigma_t$ are the posterior mean and variance from surrogate model, and $\beta_t$ is a hyperparameter that controls exploration and exploitation of the acquisition function. The next sample point proposed by acquisition function is then evaluated on the objective function, adding to the sample data set for next round of surrogate modelling and acquisition evaluation. Gaussian process regression and evaluation on acquisition function are performed repeatedly during the optimization process until a pre-defined budget on number of iterations is achieved.

\subsection{Mixed-Integer Quadratic Programming}\label{subsec:MIQP}

Mixed-integer quadratic programming (MIQP) considers problems with quadratic constraints/objectives:
\begin{equation*}
    \begin{aligned}
        \min\limits_{\bm{x}\in\mathbb R^D}~& \bm{x^TQx}+\bm{q^Tx}
        && \text{(objective)}\\
        s.t.~& \bm{Ax}\le\bm{b}
        && \text{(linear)}\\
        &\bm{x^TQ_jx}+\bm{q_j^Tx}\le b_j,~\forall j\in \mathcal J 
        && \text{(quadratic)}\\
        &x_i\in\mathbb Z,~\forall i\in\mathcal I\subset\{1,\dots, D\}
        && \text{(integrality)}\\
        &\bm{x}_L\le\bm{x}\le\bm{x}_U
        && \text{(bounds)}\\
    \end{aligned}
\end{equation*}

Section \ref{subsec:MIQP_formulation} presents our MIQP formulation for LCB with approximated posterior mean and variance. Several commercial solvers for MIQP are available; we use Gurobi v11.0.0 \citep{gurobi2024}.

\section{METHODOLOGY}\label{sec:methodology}

\subsection{Piecewise linearization of kernel}\label{subsec:pwl_kernel}

This section introduces our proposed piecewise-linear approximation for stationary or dot-product kernel functions. Consider a kernel function $k(r): \mathbb{R}^+_0 \rightarrow \mathbb{R}^+$, where $r$ is some distance measure or dot-product between data points $\bm{x}, \bm{x'}\in\mathbb{R}^D$. Intuitively, fewer linear segments are needed for approximately-linear parts of the function. We therefore select piecewise-linear breakpoints based on the curvature of $k(\cdot)$, which is proportional to its second derivative. Specifically, we first set a threshold $\epsilon_k$ to define ``near-linear'' parts, i.e., segments  of the kernel with $|k''(r)|\leq \epsilon_k$. 

The threshold value can be selected to balance the trade-off between the number of piecewise-linear segments and the approximation accuracy. We empirically choose $\epsilon_k$ to be half of the maximal value of $k''(r)$, which bounds the approximation error $e_{approx}$ to relatively small values. Future work could investigate a more systematic derivation of this threshold value. Figure \ref{pwl_illus} visualizes our piecewise linearization strategy using the Mat{\'e}rn 3/2 kernel as an example.

\begin{figure}[!htb]
  \includegraphics[width=\columnwidth]{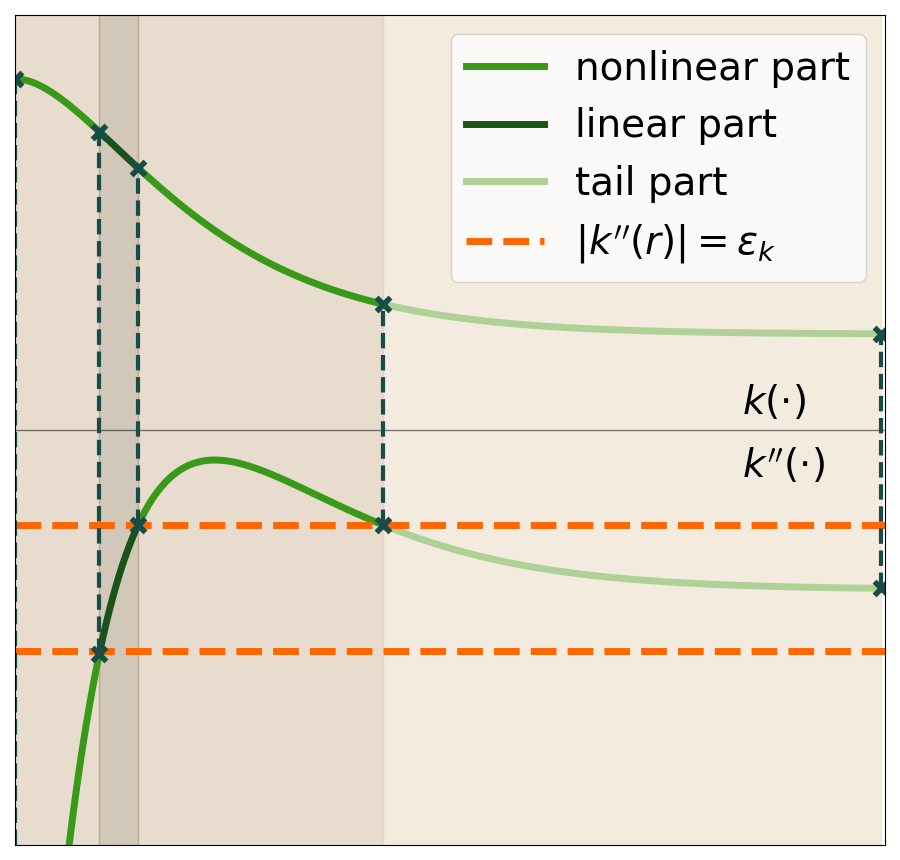}
  \caption{(\textbf{top}) Mat{\'e}rn 3/2 kernel function divided into 3 parts. (\textbf{bottom}) The second-order derivative of Mat{\'e}rn 3/2 kernel function. Parts within threshold are considered as ``near-linear.''}
  \label{pwl_illus}
\end{figure}

For stationary kernels enjoying similar shape to the Mat{\'e}rn 3/2 kernel function, e.g., RBF, Mat{\'e}rn 5/2, the kernel is partitioned into three parts:
\begin{equation*}
    \underbrace{[r_0,r_1]\cup [r_2,r_3]}_{R_\mathrm{nonlinear}} \cup \underbrace{[r_1,r_2]}_{R_\mathrm{linear}} \cup \underbrace{[r_3,r_{4}]}_{R_\mathrm{tail}}
\end{equation*}
where $r_1, r_2, r_3$ are derived by solving $|k''(r) | = \epsilon_k$, and $r_0, r_{4}$ denote the minimum and maximum values of $r$.

Based on the above, we use $D$ segments in the domain of $R_\mathrm{linear}$ and $2D$ segments for both $R_\mathrm{nonlinear}$ and $R_\mathrm{tail}$. We use $2D$ segments in $R_{tail}$ due to its relatively large range. The final set of breakpoints is:
\begin{equation*}
    \begin{aligned}
        R&=S_{r_0,r_1}^{2D}\cup S_{r_1,r_2}^D\cup S_{r_2,r_3}^{2D}\cup S_{r_3,r_4}^{2D}\cup\{r_4\}\\
        &=\{R_0,\cdots, R_M~|~R_i<R_j,~\forall i < j\}
    \end{aligned}
\end{equation*}
where $S_{l,r}^n$ denotes the set of $n$ points evenly spaced over interval $[l,r)$, and $M$ is the number of segments.

Using the approximated kernel, denoted by $\tilde k(\cdot)$, we can now define the approximated posterior mean $\tilde{\mu}$ and variance $\tilde{\sigma}^2$ analogously to \eqref{eq:GP_mean_variance}:
\begin{equation}\label{eq:GP_approximated_mean_variance}
    \begin{aligned}
    \tilde{\mu}(\bm{x})&:=\tilde{K}_{xX}\tilde{K}_{XX}^{-1}\bm{y}\\
    \tilde{\sigma}^2(\bm{x})&:=\tilde K_{xx}-\tilde{K}_{xX} \tilde{K}_{XX}^{-1}\tilde{K}_{Xx}
    \end{aligned}
\end{equation}
which then give us the approximated LCB:
\begin{equation}\label{eq:approx_LCB}
    \begin{aligned}
        \tilde\alpha_\mathit{LCB}(\bm{x})=\tilde \mu(\bm{x})-\beta_t^{1/2}\tilde \sigma(\bm{x})
    \end{aligned}
\end{equation}
where by convention we take $\beta_t = 0.2 D \log{2t}$ for its convergence properties~\citep{Kandasamy2015AGP}.

\begin{table}[ht!]
    \centering
    \caption{Piecewise linearization parameter values for Mat{\'e}rn 3/2 and RBF kernels.}
    \label{tab:pwl_params}
    \begin{tabular}{ccc}
        \toprule
        Variables & Mat{\'e}rn 3/2 & RBF \\
        \midrule
        $\epsilon_k$ & $\frac{3}{2}\exp(-2)\sigma_f^2$ & $\exp(-\frac{3}{2}) \sigma_f^2$\\
        $e_{approx}$ & $0.025 \sigma_f^2$ & $0.022\sigma_f^2$\\
        $r_1$ & 0.4866 & 0.8280\\
        $r_2$ & 0.7113 &1.2099\\
        $r_3$ & 2.1237 & 2.5213\\
        \bottomrule
    \end{tabular}
\end{table}

We take two classic stationary kernels, Mat{\'e}rn 3/2 and RBF kernel, as examples. Their kernel functions are:
\begin{equation*}
    \begin{aligned}
        k_\mathit{Mat\text{{\'e}}rn\ 3/2}(r) &= \sigma_f^2(1+\sqrt{3}r)\exp(-\sqrt{3}r)\\
        k_\mathit{RBF}(r) &= \sigma_f^2\exp(-\frac{1}{2}r^2)
    \end{aligned}
\end{equation*}
where $r=\frac{\|\bm{x}-\bm{x'}\|_2}{l}$ and $l$ is the lengthscale.

The values of parameters involved in their piecewise-linear approximations are presented in Table \ref{tab:pwl_params}.

\subsection{Mixed-integer Optimization Formulation}\label{subsec:MIQP_formulation}
Given this piecewise linear kernel approximation, we now formulate the minimization of the approximated LCB \eqref{eq:approx_LCB} as an MIQP:

\begin{subequations}\label{eq:final_model}
    \begin{align}
        \min\limits_{x\in\mathcal X}~&
        \tilde{\mu}-\beta_t^{1/2}\tilde{\sigma}\label{eq:line_a}\\
        s.t.~& \tilde{\mu} = \tilde{K}_{xX}\tilde{K}_{XX}^{-1}\bm{y}\label{eq:line_b}\\
        &\tilde{\sigma}^2\le\tilde{\sigma}_f^2-\tilde{K}_{xX} \tilde{K}_{XX}^{-1}\tilde{K}_{Xx}\label{eq:line_c}\\
        &r_i^2=\frac{\|\bm{x}-\bm{x_i}\|_2^2}{l^2},~\forall 1\le i\le N\label{eq:line_d} \\
        &\tilde K_{xX_i}=\tilde k(r_i),~\forall 1\le i\le N\label{eq:line_e}
    \end{align}
\end{subequations}

Note $\beta_t^{1/2},\tilde{K}_{XX}^{-1},\bm{y},l$ are independent of $\bm x$ and their values are precomputed. Constraints \eqref{eq:line_b} and \eqref{eq:line_c} follow from \eqref{eq:GP_approximated_mean_variance}, where the inequality comes from quadratic constraint relaxation. The scaled Euclidean distance is given by constraint \eqref{eq:line_d}. 

Constraint \eqref{eq:line_e} involves the piecewise linearization of $k(\cdot)$, whose encoding is well-studied in MIP literature \citep{Beale1969sos1, Forrest1974largeMIP} and is provided in modern MIP solvers such as Gurobi \citep{gurobi2024}. Here we present the classic encoding using our notations in Eq.~\eqref{eq:pwl_encoding}. First, as shown in Eqs.~\eqref{eq:pwlline_a}--\eqref{eq:pwlline_c}, point $(r_i,\tilde k(r_i))$ is expressed as a convex combination of points $\{(R_j,\tilde k(R_j))\}_{0\le j\le M}$ with $\{w_i^j\}_{0\le j\le M}$ as the nonnegative and sum-to-one coefficients. Then, as given in Eqs.~\eqref{eq:pwlline_d}--\eqref{eq:pwlline_g}, an indicator variable $\lambda_i^j$ is introduced to model that point $(r_i,\tilde k(r_i))$ lies in $j$-th linear segment between $(R_{j-1},\tilde k(R_{j-1}))$ and $(R_j,\tilde k(R_j))$, i.e., only $w_i^{j-1}$ and $w_i^j$ can be nonzero if $\lambda_i^j=1$.
\begin{subequations}\label{eq:pwl_encoding}
    \begin{align}
        &r_i=\sum\limits_{j=0}^Mw_i^jR_j\label{eq:pwlline_a}\\
        &\tilde K_{xX_i}=\sum\limits_{j=0}^Mw_i^j\cdot \tilde{k}(R_j)\label{eq:pwlline_b}\\
        &\sum_{j=0}^{M}w_i^j=1,~w_i^j\ge 0,~\forall 0\le j\le M\label{eq:pwlline_c}\\
        &\sum\limits_{j=1}^M\lambda_i^j=1,~\lambda_i^j\in\{0,1\},~\forall 1\le j\le M\label{eq:pwlline_d}\\
        &w_i^0\le \lambda_i^1\label{eq:pwlline_e}\\
        &w_i^j\le \lambda_i^j+\lambda_i^{j+1},~\forall 1\le j<M\label{eq:pwlline_f}\\
        &w_i^M\le \lambda_i^M\label{eq:pwlline_g}
    \end{align}
\end{subequations}




\textbf{Initialization Heuristic.}
Note that finding a feasible initial point (upper bound) for \eqref{eq:final_model} is simple: we need only select a point $\boldsymbol{x}$ and evaluate \eqref{eq:GP_approximated_mean_variance}--\eqref{eq:approx_LCB}. 
We propose a heuristic to find initial points based on minimizing the posterior mean. Specifically, before solving \eqref{eq:final_model}, PK-MIQP first minimizes a sub-problem with $\tilde{\mu}$ as the objective and \eqref{eq:line_b}, \eqref{eq:line_d}--\eqref{eq:line_e} as constraints. Without the quadratic constraint \eqref{eq:line_c}, the sub-problem can be solved relatively quickly, producing an initial solution pool $P_\mathrm{sub}$ containing solutions with lowest mean. The full-problem \eqref{eq:final_model} is then solved using the best solution among $P_\mathrm{sub}$ and several randomly sampled points $P_\mathrm{rand}$ as a good incumbent solution. 
While the approximation error is theoretically analyzed in the following section, in practice we employ two steps to polish the solution. First, PK-MIQP selects the best solution found while solving \eqref{eq:final_model} with lowest $\alpha(\cdot)$, i.e., true LCB value. Then we apply a few steps of gradient descent to ensure we are exactly at a local optimum. The final solution $\bm{x_t}$ obtained is a minimum with global optimality guarantees (see Section~\ref{subsec:theory}). PK-MIQP can seamlessly handle \textit{constrained} optimization problems by adding known constraints to the formulation \eqref{eq:final_model}. PK-MIQP also adapts to any acquisition function that can be linearly (or quadratically) represented, e.g., GLCB \citep{rodemann2024glcb}, by replacing the objective function in \eqref{eq:line_a} accordingly. One iteration of PK-MIQP is outlined in Algorithm \ref{final_alg}.

\begin{algorithm}[tb]
    \caption{PK-MIQP at $t$-th iteration}
    \label{final_alg}
    \begin{algorithmic}
        \STATE {\bfseries Input:} sample points $\mathcal D_{t-1} =\{(\bm{x_i}, y_i)\}_{i=1}^{t-1}$, $\beta_t$.

        \STATE {\bfseries Kernel approximation:}
        
        \STATE\quad kernel parameters $\sigma_f^2,~l\gets$ GP fit to $\mathcal D_{t-1}$.
        
        \STATE\quad $\tilde k(\cdot)\gets$ piecewise linearization. 

        \STATE {\bfseries Warm start (optional):}
        \STATE\quad $P_\mathrm{sub}\gets$ solve sub-problem of \eqref{eq:final_model}.

        \STATE\quad $P_\mathrm{rand}\gets$ random feasible solutions of \eqref{eq:final_model}.

        \STATE {\bfseries Acquisition optimization:}

        \STATE\quad $P_\mathrm{full}\gets$ solve \eqref{eq:final_model} (incumbent from $P_\mathrm{sub} \cup P_\mathrm{rand}$).
        
        \STATE {\bfseries Solution polishing:}

        \STATE\quad $\bm{x_t^0}=\mathop{\arg\min}\limits_{{\bm x}\in P_\mathrm{full}}\alpha(\bm{x})\gets$ solution with lowest LCB.

        \STATE\quad $\bm{x_t}\gets$ correction starting at $\bm{x_t^0}$.
        
        \STATE {\bfseries Output:} next sample $\bm{x_t}$.
    \end{algorithmic}
\end{algorithm}


As PK-MIQP is a generic optimization framework, computational cost can be reduced using existing strategies, e.g., additive GP (add-GP) training \citep{duvenaud2011additiveGP, Kandasamy2015AGP}. The basic idea of add-GP is to decompose black-box function $f$ among $N_g$ disjoint sets of dimensions:
\begin{equation*}
        f(\bm{x}) = \sum_{i = 1}^{N_g} f^{(i)}(\bm{x}^{(i)})
\end{equation*}
where each set of dimensions has a kernel that only acts on the included dimensions. 
Since the sets of dimensions are independent, they can be optimized separately, e.g., applying PK-MIQP $N_g$ times in parallel. 

\subsection{Theoretical Analysis}\label{subsec:theory}

Given the approximation and methodology above, we now aim to establish a bound on regret $r_t:=f(\bm{x_t}) - f(\bm{x^*})$ for using a GP with the proposed approximated kernel at the $t^{th}$ iteration of a full BO loop. Here $f$ is the black-box objective, $\bm{x^*}$ denotes the (oracle) optimum point that minimizes $f$, and $\bm{x_t}$ is the chosen point to evaluate at iteration $t$. 
We begin with two remarks on the proposed approximated kernel:

\begin{remark}
    The difference between the true and approximated kernel functions is bounded by error $\epsilon_M$, which is a function of the number of linear pieces $M$.
\end{remark}

\begin{remark}
    Error $\epsilon_M$ asymptotically converges to 0 as the number of linear pieces increases:
\begin{equation*}
    \begin{aligned}
        \lim_{M\to\infty} \epsilon_M = 0
    \end{aligned}
\end{equation*}
\end{remark}

With the above remarks, Theorem \ref{thm:main} shows that the approximated kernel mean (and variance) converges to the true mean (and variance) as $M\to\infty$.

\begin{theorem}\label{thm:main}
    Given $N$ observed data points $\bm{X}$ with outputs $\bm{y}$, for any ${\bm x}\in\mathcal D$, we have:
    \begin{equation*}
        \begin{aligned}
            |\mu(\bm{x})-\tilde\mu(\bm{x})|\le C_{\mu}N^2\epsilon_M,~|\sigma(\bm{x})-\tilde\sigma(\bm{x})|\le C_{\sigma}N\epsilon_M^{1/2}\\
        \end{aligned}
    \end{equation*}
\end{theorem}
\begin{proof}[Proof (Sketch)]
Denote $\psi:=K_{xX}-\tilde K_{xX}$ and $\Psi:=K_{XX}-\tilde K_{XX}$. By the definition of $\mu(\bm{x})$ and $\tilde\mu(\bm{x})$, 
    \begin{equation*}
        \begin{aligned}
            \mu(\bm{x})-\tilde\mu(\bm{x}) 
            =-\tilde{K}_{xX}K_{XX}^{-1}\Psi\tilde{K}_{XX}^{-1} \bm{y}+\psi K_{XX}^{-1}\bm{y}
        \end{aligned}
    \end{equation*}
Since $\|\psi\|_2\le \sqrt{N}\epsilon_M,~\|\Psi\|_2\le N\epsilon_M,~\|\tilde K_{xX}\|_2\le \sqrt{N}\sigma_f^2$, and $\|\bm{y}\|_2\le \sqrt{N}$ (In our implementation, we scaled the objective $f(x)\in [0,1]$), we have:
\begin{equation*}
    \begin{aligned}
        |\mu(\bm{x})-\tilde\mu(\bm{x})|\le C_{\mu}N^2\epsilon_M
    \end{aligned}
\end{equation*}
where the constant $C_{\mu}=\sigma^{-4}_\epsilon \sigma^2_{\max}$ and $\sigma_{\max}^2$ is the upper bound of  kernel variance $\sigma_f^2$.

Similarly, the difference of variances is:
\begin{equation*}
    \begin{aligned}
        \sigma^2(\bm{x})-\tilde\sigma^2(\bm{x})
    =&\tilde{K}_{xX}K_{XX}^{-1}\Psi\tilde{K}_{XX}^{-1} \tilde{K}_{Xx}\\
    &-2\tilde{K}_{xX}K_{XX}^{-1}\psi^T-\psi K_{XX}^{-1}\psi^T
    \end{aligned}
\end{equation*}
Then we have:
\begin{equation*}
    \begin{aligned}
        |\sigma^2(\bm{x})-\tilde\sigma^2(\bm{x})| \le C^2_\sigma N^2\epsilon_M
    \end{aligned}
\end{equation*}
where the constant $C_\sigma=\sigma^{-2}_\epsilon\sigma^2_{\max}$. Since variances are non-negative, we can derive:
\begin{equation*}
    \begin{aligned}
        |\sigma(\bm {x})-\tilde\sigma(\bm {x})|\le C_\sigma N\epsilon^{1/2}_M
    \end{aligned}
\end{equation*}
which completes the proof. A more comprehensive proof is provided in Appendix.
\end{proof}

Theorem \ref{thm:main} holds for any continuous kernel with a proper piecewise linear approximation. With additional smoothness assumptions on the kernel, we can derive similar regret bounds as \citet{Srinivas2012GPregret}. Following the same settings as Lemma 5.8 in \citet{Srinivas2012GPregret}, Theorem \ref{thm:regret} bounds the regret $r_t=f(\bm{x_t})-f({\bm x^*})$, where $\bm {x_t}$ is the optimal solution of  MIQP \eqref{eq:final_model} at $t$-th iteration. Since we may need more pieces in our approximation, denote $M_t$ as the number of linear pieces at $t$-th iteration.

\begin{theorem}\label{thm:regret}
Let $D\subset [0,r]^d$ be compact and complex, $d\in \mathbb N, r>0$. Suppose kernel $K(\bm{x},\bm{x'})$ satisfies the following high probability bound on the derivatives of GP sample paths $f$: for some constants $a,b>0$:
\begin{equation*}
    \begin{aligned}
        Pr\left\{\sup_{\bm{x}\in D}|\partial f/\partial x_j|>L\right\}\le ae^{-(L/b)^2},~j=1,2,\dots,d
    \end{aligned}
\end{equation*}
Select $\delta\in (0,1)$, and define $\beta_t=2\log(2t^2\pi^2/(3\delta))+2d\log(t^2dbr\sqrt{\log(4da/\delta)})$. Then the following regret bounds
    \begin{equation*}
        \begin{aligned}
            r_t\le 2\beta_t^{1/2}\sigma_{t-1}(\bm {x_t})+1/t^2+2C_{\mu}t^2\epsilon_{M_t}+4C_\sigma \beta_t^{1/2}t\epsilon^{1/2}_{M_t}
        \end{aligned}
    \end{equation*}
    hold with probability $\ge 1-\delta$.
\end{theorem}
\begin{proof}[Proof (Sketch)]
    Lemma 5.5 and Lemma 5.7 in \citet{Srinivas2012GPregret} give us:
    \begin{equation*}
        \begin{aligned}
            |f(\bm {x_t}) - \mu_{t-1}(\bm{x_t})|&\le \beta_t^{1/2}\sigma_{t-1}(\bm{x_t})\\
            |f(\bm{x^*})-\mu_{t-1}(\bm{[x^*]_t})|&\le\beta_{t-1}(\bm{[x^*]_t})+1/t^2
        \end{aligned}
    \end{equation*}
    where $\bm{[x^*]_t}$ is the closest point in $D_t$ to $\bm{x^*}$, and $D_t\subset D$ is a discretization at $t$-th iteration.

    Recall the definition of $\bm{x_t}$, we have:
    \begin{multline*}
            \tilde \mu_{t-1}(\bm {x_t}) - \beta_t^{1/2}\tilde\sigma_{t-1}(\bm{x_t})\le  \\ \tilde \mu_{t-1}(\bm{[x^*]_t}) - \beta_t^{1/2}\tilde\sigma_{t-1}(\bm{[x^*]_t})
    \end{multline*}
    Using Theorem \ref{thm:main} to replace $\mu,~\sigma$ with $\tilde\mu,~\tilde\sigma$, we have:
    \begin{equation*}
        \begin{aligned}
            r_t=&~f(\bm{x_t})-f(\bm{x^*})\\
            \le&~\tilde\mu_{t-1}(\bm{x_t})-\tilde\mu_{t-1}(\bm{[x^*]_t})+\beta_t^{1/2}\tilde\sigma_{t-1}(\bm{[x^*]_t})\\
            &+\beta_t^{1/2}\tilde\sigma_{t-1}(\bm{x_t})+1/t^2+2C_{\mu}t^2\epsilon_M+2C_\sigma \beta_t^{1/2}t\epsilon^{1/2}_M\\
            \le&~2\beta_t^{1/2}\tilde\sigma_{t-1}(\bm{x_t})+1/t^2+2C_{\mu}t^2\epsilon_{M_t}+2C_\sigma \beta_t^{1/2}t\epsilon^{1/2}_{M_t}
        \end{aligned}
    \end{equation*}
    Applying Theorem \ref{thm:main} again completes this proof. 
\end{proof}

See Appendix \ref{app:full_proof} for full proofs of Theorems \ref{thm:main} and \ref{thm:regret}. Lemma \ref{lemma:bounds} is a direct conclusion from Theorem \ref{thm:regret}, which implies that the same regret bounds as \citet{Srinivas2012GPregret} hold when using PK-MIQP to optimize approximated LCB.

\begin{lemma}\label{lemma:bounds}
    For $\epsilon_{M_t}=1/\mathcal O(t^{4+\epsilon}),~\forall t\ge 1$ with some $\epsilon>0$, running PK-MIQP for a sample $f$ from a GP with zero mean and covariance $K(\bm{x},\bm{x'})$, we can bound regret by $\mathcal O^*(\sqrt{dT\gamma_T})$ with high probability. Precisely, pick $\delta\in(0,1)$, with $C_1=8/\log(1+\sigma_f^{-2}),~C_2=\sum\limits_{t=1}^T(1/t^2+2C_{\mu}t^2\epsilon_{M_t}+4C_\sigma \beta_t^{1/2}t\epsilon^{1/2}_{M_t})$, we have:
    \begin{equation*}
        \begin{aligned}
            Pr\left\{R_T\le \sqrt{C_1T\beta_T\gamma_T}+C_2\right\}\ge 1-\delta
        \end{aligned}
    \end{equation*}
    where $R_T=\sum_{t=1}^Tr_t$ is the cumulative regret, $\gamma_T$ is the maximal information gain after $T$ rounds.
\end{lemma}
\begin{proof}
    Replacing the term $1/t^2$ by $1/t^2+2C_{\mu}t^2\epsilon_{M_t}+4C_\sigma \beta_t^{1/2}t\epsilon^{1/2}_{M_t}$ in the proof of Theorem 2 in \citet{Srinivas2012GPregret} finishes this proof.
\end{proof}

\begin{remark}
    The smoothness assumption holds for any stationary kernel $K(\bm{x},\bm{x'})=k(\bm{x}-\bm{x'})$ that is four times differentiable, including squared exponential and Mat{\'e}rn kernels with $\nu>2$. Whether the same conclusion for Mat{\'e}rn 3/2 holds is unclear, so we choose Mat{\'e}rn 3/2 in the main paper to empirically test its performance and find promising results. 
\end{remark}

\section{RESULTS}\label{sec:results}

\begin{table*}[h!]
    \centering
    \caption{Comparison of solvers on optimizing random acquisition functions using Mat{\'e}rn 3/2 kernel. The mean of the optimal LCB values found over the 20 replications is reported with 0.5 standard deviation in parentheses. PK-MIQP consistently outperforms other gradient- and sampling-based methods.}
    \label{tab:one_iter}
    \begin{tabular}{cccccc}
    \toprule
    method & 1D & 2D & 3D & 4D & 5D \\
    \midrule
    L-BFGS-B & -0.68(0.54) & -0.94(0.51) & -1.79(0.74) & -1.21(0.77) & -1.22(0.50)\\
    Nelder-Mead & -0.65(0.51) & -0.96(0.51) & -1.32(0.42) & -1.98(1.07) & -1.07(0.41)\\
    COBYLA & -1.09(0.44) & -1.26(0.52) & -1.48(0.75) & -1.58(0.75) & -0.68(0.49)\\
    SLSQP & -0.59(0.53) & -0.99(0.70) & -1.58(0.75) & -1.49(0.86) & -1.43(0.51)\\
    trust-constr & -0.54(0.53) & -1.19(0.55) & -1.95(0.53) & -1.71(0.88) & -1.61(0.49)\\
    PK-MIQP & \textbf{-1.76(0.27)} & \textbf{-2.26(0.49)} & \textbf{-2.17(0.36)} & \textbf{-2.25(0.74)} & \textbf{-1.62(0.50)}\\
    \bottomrule
    \end{tabular}
\end{table*}

\begin{table*}[h!]
    \centering
    \caption{Comparison of solvers on optimizing random acquisition functions using Mat{\'e}rn 3/2 kernel. The mean of the computational time in seconds over the 20 replications is reported with 0.5 standard deviation in parentheses.}
    \label{tab:comp_time}
    \begin{tabular}{cccccc}
    \toprule
    method & 1D & 2D & 3D & 4D & 5D \\
    \midrule
    L-BFGS-B & 0.12(0.06) & 0.14(0.08) & 0.46(0.25) & 0.41(0.16) & 0.80(0.37)\\
    Nelder-Mead & 0.27(0.11) & 1.63(2.35) & 1.96(1.28) & 2.34(0.96) & 3.75(1.51)\\
    COBYLA & 0.21(0.06) & 0.48(0.29) & 1.57(1.32) & 2.21(1.52) & 1.68(0.75)\\
    SLSQP & 0.08(0.04) & 0.10(0.05) & 0.40(0.17) & 0.24(0.13) & 0.61(0.23)\\
    trust-constr & 0.26(0.11) & 0.78(0.21) & 0.88(0.29) & 2.67(1.76) & 1.78(0.84)\\
    add-GP PK-MIQP & 2.08(0.35) & 7.14(0.93) & 15.41(2.04) & 19.06(4.52) & 22.45(3.34)\\
    PK-MIQP & 2.08(0.35) & 11.37(3.95) & 866.48(601.49) & 1557.82(576.70) & 3199.24(361.09)\\
    \bottomrule
    \end{tabular}
\end{table*}

We compare the performance of PK-MIQP against state-of-the-art minimizers over several benchmarks and a hyperparameter tuning problem. 
The Mat\'ern 3/2 kernel and LCB are used for all methods; only the optimizer used for the acquisition function is changed. For gradient-based methods, we choose L-BFGS-B \citep{Zhu1997LBFGS}, SLSQP \citep{Kraft1988slsqp} and trust-constr \citep{Byrd1999trustconstr}. For sampling-based methods, we consider COBYLA \citep{Powell1994cobyla} and Nelder-Mead \citep{Nelder1965neldermead}. These methods are chosen since they are relatively general and can be directly deployed without altering the BO algorithm or acquisition function. We use the default implementations in \texttt{scipy} \citep{Scipy2020sciPy}.

All experiments were performed on a 3.2 GHz Intel Core i7-8700 CPU with 16 GB memory. For each case, we report the mean with $\pm 0.5$ standard deviation of simple regret over $20$ replications. For all benchmarks, we initially sample $\min\{10D,30\}$ points using Latin hypercube sampling (LHS). For the SVM tuning problem, the size of the initial set is chosen as $10$, since evaluations are time-consuming. At each iteration, the sample points are standardized to $[0, 1]$ before the GP training. We set size of solution pools $P_\mathrm{sub}$ and $P_\mathrm{rand}$ to 10 for PK-MIQP. We use GPflow \citep{Matthews2017GPflow} to implement GP models and Gurobi v11.0.0 \citep{gurobi2024} to solve the resulting MIQPs (including solution pools). Full implementation details are provided in Appendix \ref{app:exp_details}. The code is available at \href{https://github.com/YilinElinXie}{GitHub}.

\subsection{Single acquisition function optimization}\label{subsec:BO_one_step}

Before considering a full BO loop, we first study the performance of PK-MIQP on optimizing a given acquisition function. Specifically, we consider GP models given random samples from the prior in 1D--5D and employ different solvers to minimize the resulting LCB. The results for GP models with  Mat{\'e}rn 3/2 kernel are given in Table \ref{tab:one_iter}. PK-MIQP consistently outperforms the other gradient- and sampling-based methods considered, suggesting the benefits of global optimization. Notably, increasing the problem dimension does not significantly impact the solution quality returned by PK-MIQP. The same conclusions can be made when using the RBF kernel (results given in Appendix \ref{app:rbf}), which furthermore supports the generalizability of the proposed PK-MIQP framework.

\subsection{Real-time results}\label{subsec:time_results}
To investigate the computational requirements to achieve global optimality using PK-MIQP, we report the computational time of each method used in Section \ref{subsec:BO_one_step}. Add-GP training is applied to 1D--5D functions sampled from GP models, where each function is decomposed into $N_g=D$ groups. The GP model is then trained with the additive kernel, and optimization solvers are applied for each group independently.

As shown in Table \ref{tab:comp_time}, PK-MIQP has a higher time complexity compared to gradient- and sample-based methods, which is the expected cost of employing global optimization. Note that BO is particularly useful when evaluating the unknown objective function is expensive, meaning extended computational times may not be a deterrent in many BO settings. Introducing add-GP greatly reduces the time cost for PK-MIQP, and we hope that our work motivates future research in MIQP methods to accelerate solutions.

\subsection{Bayesian optimization using PK-MIQP}\label{subsec:BO_full_loop}

\begin{figure*}[ht]
     \centering
     \begin{subfigure}[b]{0.325\textwidth}
         \centering
         \includegraphics[width=\textwidth]{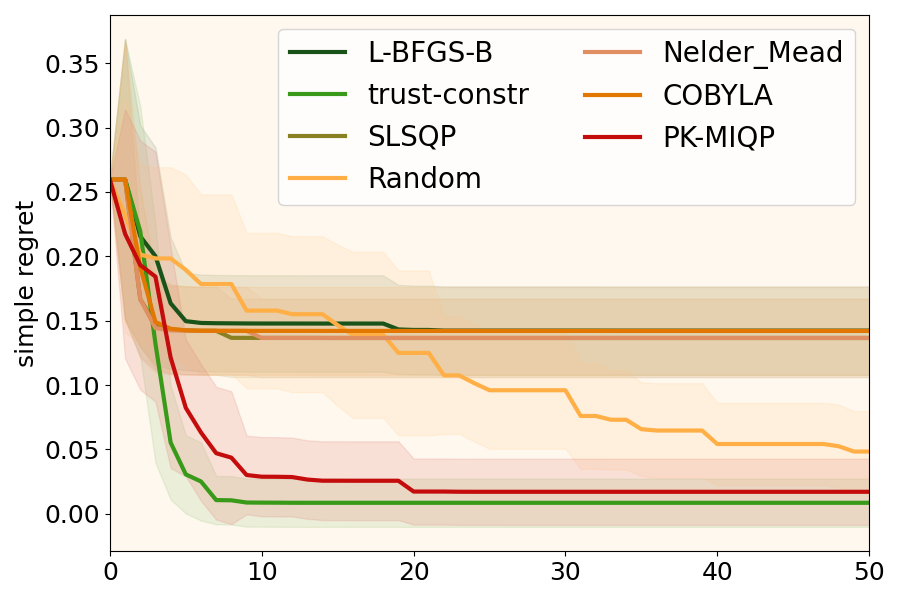}
         \caption{1D Multimodal}
         \label{fig:test_2}
     \end{subfigure}
     \begin{subfigure}[b]{0.325\textwidth}
         \centering
         \includegraphics[width=\textwidth]{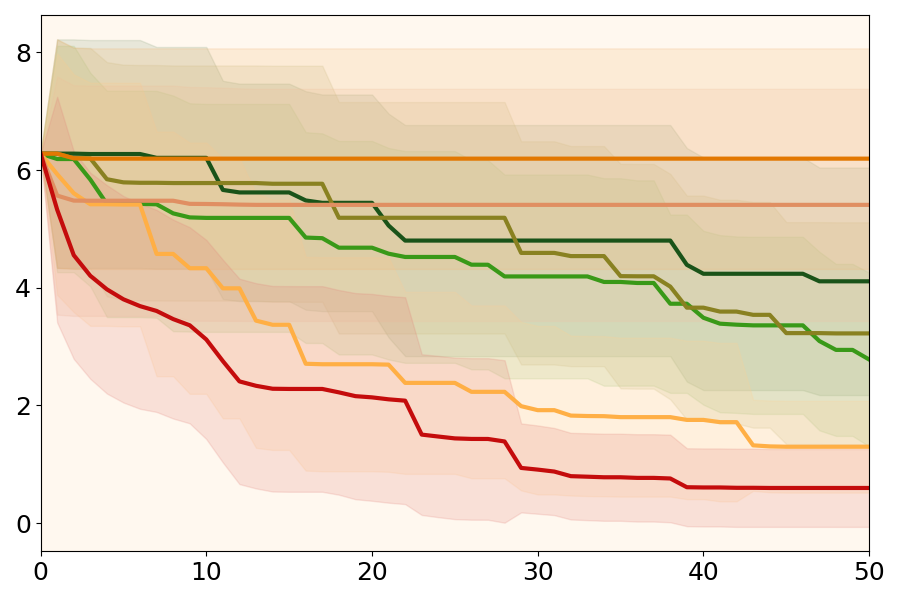}
         \caption{1D Bumpy}
         \label{fig:test_1}
     \end{subfigure}
     \begin{subfigure}[b]{0.325\textwidth}
         \centering
         \includegraphics[width=\textwidth]{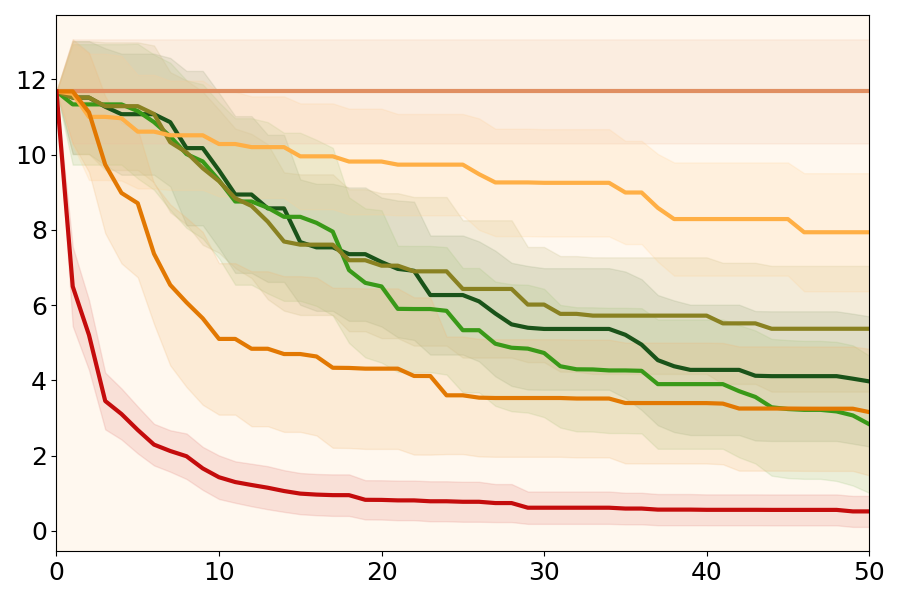}
         \caption{2D Ackley}
         \label{fig:ackley}
     \end{subfigure}
     \begin{subfigure}[b]{0.325\textwidth}
         \centering
         \includegraphics[width=\textwidth]{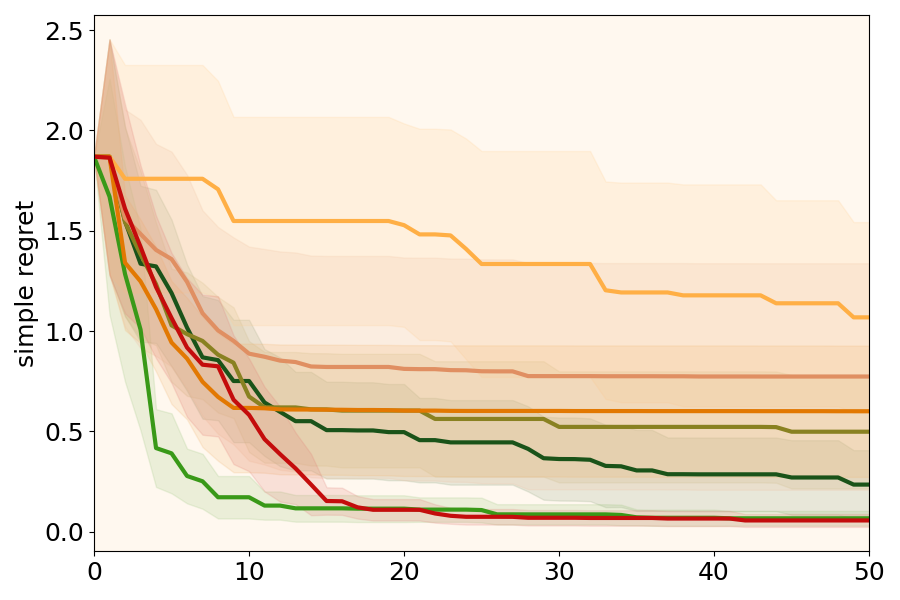}
         \caption{2D Branin}
         \label{fig:branin}
     \end{subfigure}
     \begin{subfigure}[b]{0.325\textwidth}
         \centering
         \includegraphics[width=\textwidth]{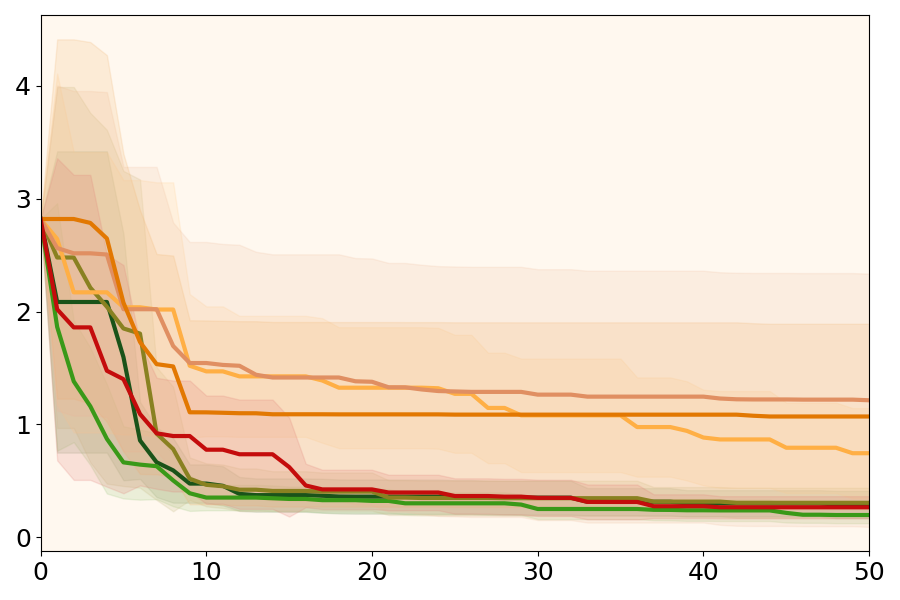}
         \caption{2D Rosenbrock}
         \label{fig:rosenbrock}
     \end{subfigure}
    \begin{subfigure}[b]{0.325\textwidth}
         \centering
         \includegraphics[width=\textwidth]{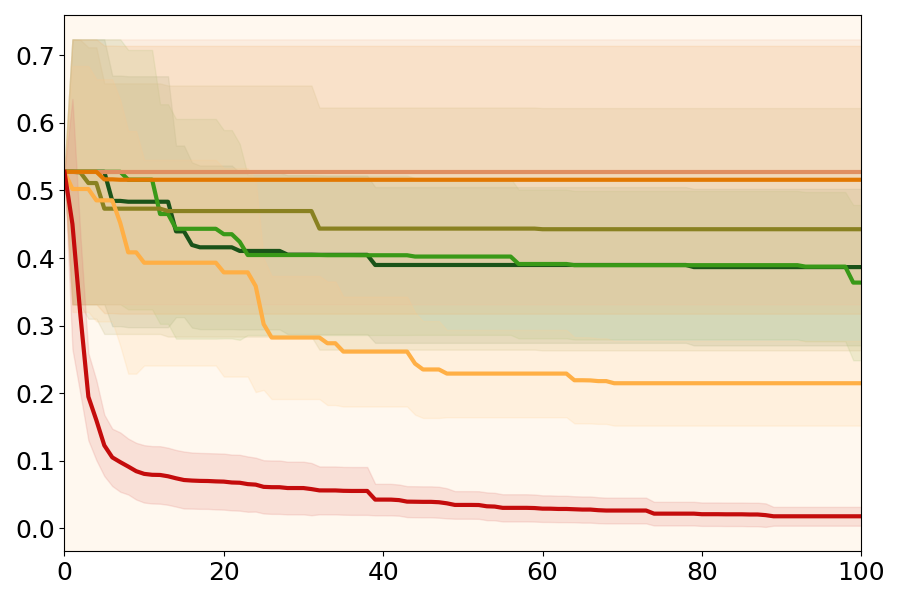}
         \caption{3D Hartmann}
         \label{fig:3D}
     \end{subfigure}
     \begin{subfigure}[b]{0.325\textwidth}
         \centering
         \includegraphics[width=\textwidth]{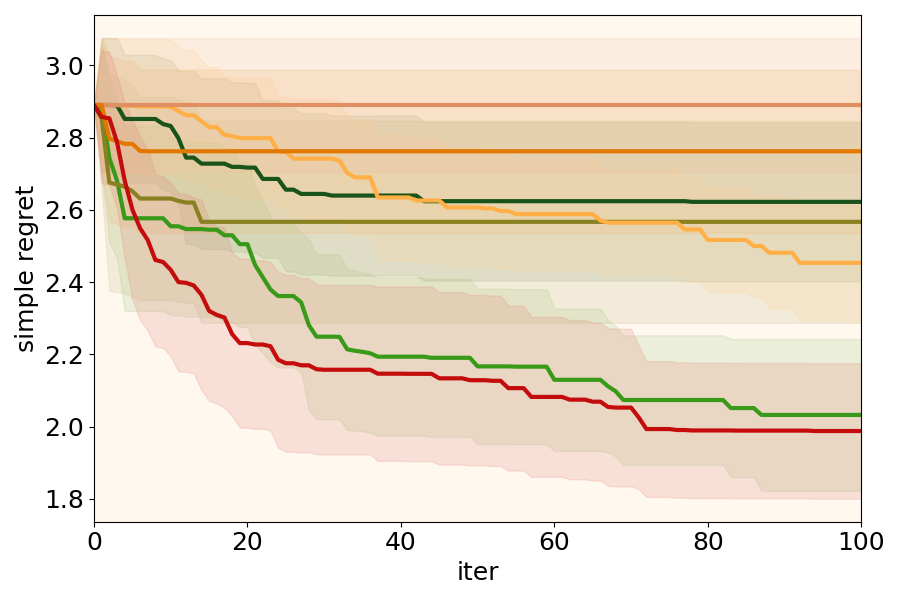}
         \caption{5D Michalewicz}
         \label{fig:5D}
     \end{subfigure}
     \begin{subfigure}[b]{0.325\textwidth}
         \centering
         \includegraphics[width=\textwidth]{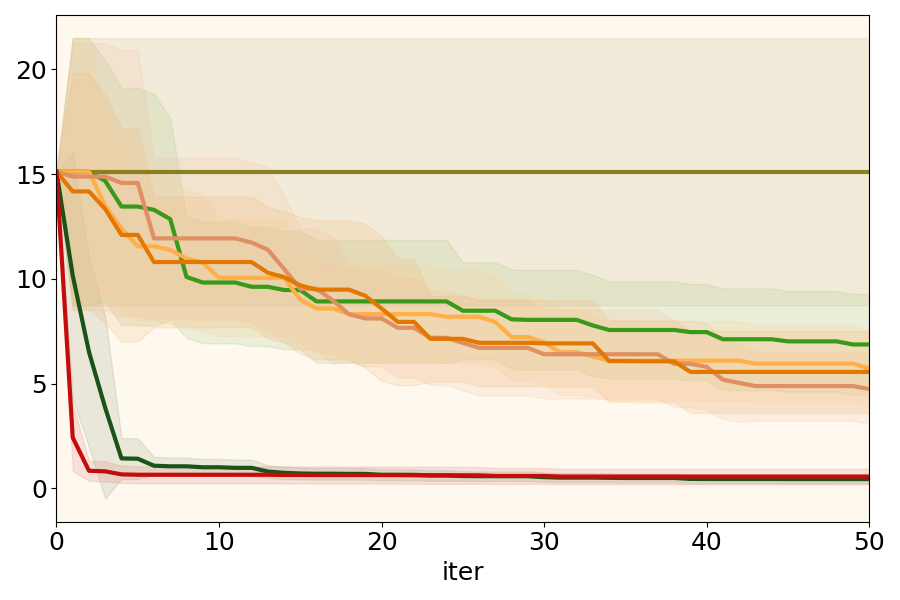}
         \caption{2D constrained KS224}
         \label{fig:constrained}
     \end{subfigure}
     \begin{subfigure}[b]{0.325\textwidth}
         \centering
         \includegraphics[width=\textwidth]{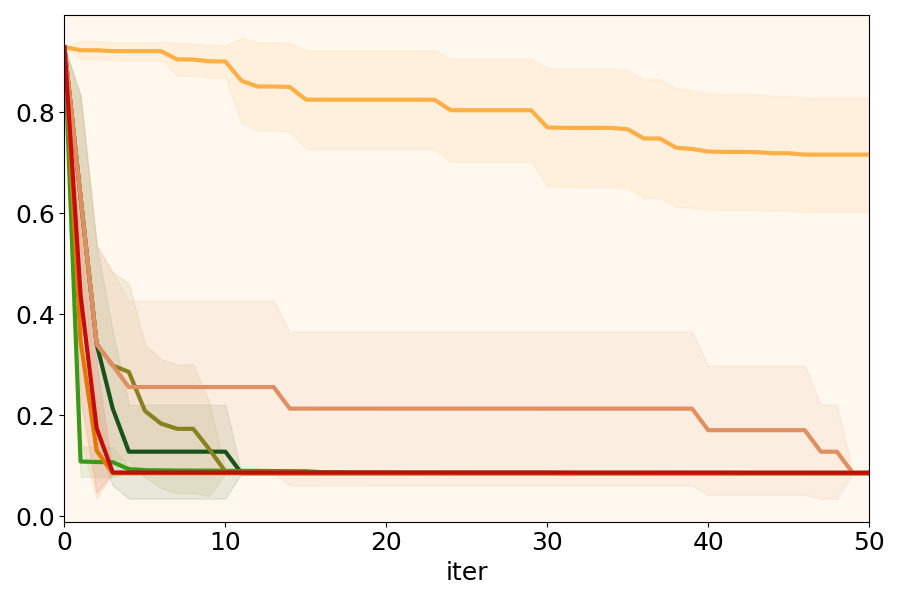}
         \caption{2D hyperparameter tuning}
         \label{fig:svm}
     \end{subfigure}
    \caption{Numerical results on Bayesian optimization using PK-MIQP with Mat{\'e}rn 3/2 kernel and the state-of-the-art minimizers. The mean with $0.5$ standard deviation of simple regret is reported over $20$ replications.}
    \label{fig:benchmark}
\end{figure*}

This section tests the performance of PK-MIQP with the Mat{\'e}rn 3/2 kernel in a full BO loop over the following functions (see Appendix \ref{app:exp_details} for their formulations):

\textbf{Unconstrained benchmarks:} Bumpy (1D), Multimodal (1D), Ackley (2D), Branin (2D), Rosenbrock (2D), Hartmannn (3D), and Michalewicz (5D). Most of these functions are commonly used as synthetic benchmarks in BO literature. Additionally, to highlight cases where intuitively there may be significant difference between choosing a global minima and a local minima of the acquisition function, we also select Bumpy and Multimodal as benchmarks. Bumpy is periodic with multiple local and global minima, and Multimodal has multiple local minima and a unique global minimum.

\textbf{Constrained benchmark:} KS224 (2D) function with $4$ linear constraints \citep{ks2024ks224}. This function is used to demonstrate the ability of PK-MIQP to handle additional constraints.

\textbf{SVM hyperparameter tuning:} As a real-world example with long evaluation times, we consider the hyperparameter tuning task using support vector machine (SVM) as a text classifier. The regularization parameter and kernel coefficient in SVM are set as the hyperparameters to be tuned. The objective is to maximize the 5-fold cross-validation score of the SVM. 

For the 3D and 5D benchmarks, we use add-GP training for all methods as in Section \ref{subsec:time_results}. Figures \ref{fig:3D} and \ref{fig:5D} show the results for the these cases. 

For the constrained benchmark, PK-MIQP can easily handle linear and quadratic constraints by adding them into formulation \eqref{eq:final_model}. For methods that cannot directly incorporate constraints, such as L-BFGS-B and Nelder-Mead, we add a penalty term to LCB on the constraint violation with a scaling parameter $\lambda$. We run the experiment three times with $\lambda$ set as $\{10, 100, 1000\}$ respectively. We pick the $\lambda$ setting that achieves the lowest regret at the end of BO, and report the results in Figure \ref{fig:constrained}. For hyperparameter tuning, we use the SVM implementation in scikit-learn \citep{Sklearn2011scikitlearn}, which is trained as a text classifier on the $20$ news group text dataset \citep{Lang1995news}.



Optimization results on the benchmark functions and real-world application are presented in Figure \ref{fig:benchmark}. BO with PK-MIQP demonstrates comparable performance to the state-of-the-art comparison methods on smooth functions, e.g., 1D multimoddal, 2D Branin and 2D Rosenbrock. On more difficult functions with numerous local minima, e.g., 1D Bumpy and 2D Ackley, PK-MIQP outperforms other methods considerably. This follows the intuition that gradient-based methods are easily trapped at local minima, and sampling-based methods can miss the global minima. PK-MIQP manages an (approximately) global optimization step. With add-GP training, PK-MIQP remains competitive performance. Note that, as dimensionality increases, sampling-methods such as COBYLA and Nelder-Mead can fail to improve through iterations. Figure \ref{fig:constrained} illustrates the results of a constrained benchmark, where again PK-MIQP is found to be effective in handling the known constraints. Finally, for the hyperparameter tuning problem, PK-MIQP achieves a cross-validation accuracy score at 91.4\% at the end of optimization, which outperforms most of other solvers.


\section{CONCLUSION}
    This work proposes PK-MIQP, a mixed-integer programming-based paradigm for \textit{global} optimization of GP-based acquisition functions. Our formulation introduces a piecewise-linear approximation for smooth GP kernels and a corresponding MIQP representation of acquisition functions. We analyze the theoretical regret bounds of PK-MIQP, and empirically demonstrate the framework on synthetic functions, constrained benchmarks, and a hyperparameter tuning task. 
    We hope this work demonstrates the potential of mixed-integer programming in BO settings. 
    Future work can improve and further scale PK-MIQP using tools from this community, e.g., cutting planes, tighter formulations, and/or branching rules. 

\section*{Acknowledgments}
The authors gratefully acknowledge support from a Department of Computing Scholarship (YX), BASF SE, Ludwigshafen am Rhein (SZ), a BASF/Royal Academy of Engineering Senior Research Fellowship (CT) and the US National Science Foundation grant 2237616 (JP). The authors also thank Alexander Thebelt, Akshay Kudva, Daniel Lengyel and Wei-Ting Tang for the helpful discussions throughout this work, as well as Jiongjian Cai for initial exploratory experiments.


\bibliographystyle{abbrvnat}
\bibliography{references}

\section*{Checklist}



 \begin{enumerate}

 \item For all models and algorithms presented, check if you include:
 \begin{enumerate}
   \item A clear description of the mathematical setting, assumptions, algorithm, and/or model. [Yes]
   \item An analysis of the properties and complexity (time, space, sample size) of any algorithm. [Yes]
   \item (Optional) Anonymized source code, with specification of all dependencies, including external libraries. [Yes]
 \end{enumerate}

 \item For any theoretical claim, check if you include:
 \begin{enumerate}
   \item Statements of the full set of assumptions of all theoretical results. [Yes]
   \item Complete proofs of all theoretical results. [Yes]
   \item Clear explanations of any assumptions. [Yes]     
 \end{enumerate}

 \item For all figures and tables that present empirical results, check if you include:
 \begin{enumerate}
   \item The code, data, and instructions needed to reproduce the main experimental results (either in the supplemental material or as a URL). [Yes]
   \item All the training details (e.g., data splits, hyperparameters, how they were chosen). [Yes]
    \item A clear definition of the specific measure or statistics and error bars (e.g., with respect to the random seed after running experiments multiple times). [Yes]
    \item A description of the computing infrastructure used. (e.g., type of GPUs, internal cluster, or cloud provider). [Yes]
 \end{enumerate}

 \item If you are using existing assets (e.g., code, data, models) or curating/releasing new assets, check if you include:
 \begin{enumerate}
   \item Citations of the creator If your work uses existing assets. [Yes]
   \item The license information of the assets, if applicable. [Not Applicable]
   \item New assets either in the supplemental material or as a URL, if applicable. [Not Applicable]
   \item Information about consent from data providers/curators. [Not Applicable]
   \item Discussion of sensible content if applicable, e.g., personally identifiable information or offensive content. [Not Applicable]
 \end{enumerate}

 \item If you used crowdsourcing or conducted research with human subjects, check if you include:
 \begin{enumerate}
   \item The full text of instructions given to participants and screenshots. [Not Applicable]
   \item Descriptions of potential participant risks, with links to Institutional Review Board (IRB) approvals if applicable. [Not Applicable]
   \item The estimated hourly wage paid to participants and the total amount spent on participant compensation. [Not Applicable]
 \end{enumerate}

 \end{enumerate}

\onecolumn
\appendix

\section{Extended Proofs}\label{app:full_proof}

\begin{proof}[Proof of Theorem 4.3]
Denote $\psi:=K_{xX}-\tilde K_{xX}$ and $\Psi:=K_{XX}-\tilde K_{XX}$. Since the difference between the true and approximated kernel functions is bounded by $\epsilon_M$ (see Remark 4.1), we can obtain:
\begin{equation*}
    \begin{aligned}
        \|\psi\|_2\le \sqrt{N}\epsilon_M,~\|\Psi\|_2\le N\epsilon_M
    \end{aligned}
\end{equation*}
where we use the fact that $\|\Psi\|_2\le N\|\Psi\|_{\max}$ and $\|\Psi\|_{\max}\le \epsilon_M$.

By the definition of $\mu(\bm{x})$ and $\tilde\mu(\bm{x})$, 
    \begin{equation*}
        \begin{aligned}
            \mu(\bm{x})-\tilde\mu(\bm{x}) 
            &= K_{xX}K_{XX}^{-1}\bm{y}-\tilde{K}_{xX}\tilde{K}_{XX}^{-1}\bm{y} \\
            &= (\tilde{K}_{xX}+\psi)K_{XX}^{-1}\bm{y}- \tilde{K}_{xX}\tilde{K}_{XX}^{-1}\bm{y} \\
            &=\tilde{K}_{xX}K_{XX}^{-1}\bm{y}+\psi K_{XX}^{-1}\bm{y}-\tilde{K}_{xX}\tilde{K}_{XX}^{-1}\bm{y} \\
            &=\tilde{K}_{xX}(K_{XX}^{-1}-\tilde{K}_{XX}^{-1}) \bm{y}+\psi K_{XX}^{-1}\bm{y} \\
            &= \tilde{K}_{xX}K_{XX}^{-1}(\tilde{K}_{XX}-K_{XX}) \tilde{K}_{XX}^{-1}\bm{y}+\psi K_{XX}^{-1}\bm{y} \\
            &= -\tilde{K}_{xX}K_{XX}^{-1}\Psi\tilde{K}_{XX}^{-1} \bm{y}+\psi K_{XX}^{-1}\bm{y}
        \end{aligned}
    \end{equation*}
where we use the property $A^{-1}-B^{-1}=A^{-1}(B-A)B^{-1}$ for two invertible matrices $A$ and $B$.

Note that $\|\tilde K_{xX}\|_2\le \sqrt{N}\sigma_f^2$ and $\|\bm{y}\|_2\le \sqrt{N}$ (recall our assumption that $f(x)\in [0,1]$), then we have:
\begin{equation*}
    \begin{aligned}
        |\mu(\bm{x})-\tilde\mu(\bm{x})|
        &\le \|\tilde K_{xX}\|_2\|K_{XX}^{-1}\|_2\|\Psi\|_2\|\tilde K_{XX}^{-1}\|_2\|\bm{y}\|_2+\|\psi\|_2\|K_{XX}^{-1}\|_2\|\bm{y}\|_2\\
        &\le N\|K_{XX}^{-1}\|_2(N\sigma_f^2\|\tilde K_{XX}^{-1}\|_2+1)\epsilon_M\\
        &\le C_{\mu}N^2\epsilon_M
    \end{aligned}
\end{equation*}
where the constant $C_{\mu}=\sigma^{-4}_\epsilon \sigma^2_{\max}$ and $\sigma_{\max}^2$ is the upper bound of  kernel variance $\sigma_f^2$. Note we ignore the small term since $N \gg 1$ and only consider the dominating term (the first term in this case).

Similarly, for the difference between variance $\sigma^2(\bm{x})$ and $\tilde\sigma^2(\bm{x})$, we have:
\begin{equation*}
    \begin{aligned}
        \sigma^2(\bm{x})-\tilde\sigma^2(\bm{x})
        &=K_{xx}-K_{xX}K_{XX}^{-1}K_{Xx}-\tilde K_{xx}+\tilde{K}_{xX}\tilde{K}_{XX}^{-1}\tilde{K}_{Xx} \\
        &=\tilde{K}_{xX}\tilde{K}_{XX}^{-1}\tilde{K}_{Xx}-K_{xX}K_{XX}^{-1}K_{Xx}  \\
        &=\tilde{K}_{xX}\tilde{K}_{XX}^{-1}\tilde{K}_{Xx}-(\tilde{K}_{xX}+\psi)K_{XX}^{-1}(\tilde{K}_{Xx}+\psi^T)\\
        &=\tilde{K}_{xX}\tilde{K}_{XX}^{-1}\tilde{K}_{Xx}-\tilde{K}_{xX}K_{XX}^{-1}\tilde{K}_{Xx}-2 \tilde{K}_{xX}K_{XX}^{-1}\psi^T-\psi K_{XX}^{-1}\psi^T \\
        &=\tilde{K}_{xX}(\tilde{K}_{XX}^{-1}-K_{XX}^{-1}) \tilde{K}_{Xx}-2\tilde{K}_{xX}K_{XX}^{-1}\psi^T-\psi K_{XX}^{-1}\psi^T\\
        &=\tilde{K}_{xX}K_{XX}^{-1}\Psi\tilde{K}_{XX}^{-1} \tilde{K}_{Xx}-2\tilde{K}_{xX}K_{XX}^{-1}\psi^T-\psi K_{XX}^{-1}\psi^T
    \end{aligned}
\end{equation*}
where $K_{xx}=\tilde K_{xx}=\sigma_f^2$ in our approximation, and we reuse equation $\tilde{K}_{XX}^{-1}-K_{XX}^{-1}=K_{XX}^{-1}\Psi\tilde{K}_{XX}^{-1}$.

Then we have:
\begin{equation*}
    \begin{aligned}
        |\sigma^2(\bm{x})-\tilde\sigma^2(\bm{x})|
        &\le \|\tilde{K}_{xX}\|_2\|K_{XX}^{-1}\|_2\|\Psi\|_2\|\tilde{K}_{XX}^{-1}\|_2 \|\tilde{K}_{Xx}\|_2+2\|\tilde{K}_{xX}\|_2\|K_{XX}^{-1}\|_2\|\psi\|_2+\|K_{XX}^{-1}\|_2\|\psi\|_2^2\\
        &\le N\|K_{XX}^{-1}\|_2(N\sigma_f^4\|\tilde K_{XX}^{-1}\|_2+2\sigma_f^2+\epsilon_M)\epsilon_M\\
        & \le C^2_\sigma N^2\epsilon_M
    \end{aligned}
\end{equation*}
where the constant $C_\sigma=\sigma^{-2}_\epsilon\sigma^2_{\max}$. 
Note that we again only consider the dominating term (the first term again).

Consider the following two cases:

\textbf{Case 1:} if $\max(\sigma(\bm {x}),\tilde\sigma(\bm {x}))\le C_\sigma N\epsilon^{1/2}_M$, then:
\begin{equation*}
    \begin{aligned}
        |\sigma(\bm {x})-\tilde\sigma(\bm {x})|
        =\max(\sigma(\bm {x}),\tilde\sigma(\bm {x}))-\min(\sigma(\bm {x}),\tilde\sigma(\bm {x}))
        \le \max(\sigma(\bm {x}),\tilde\sigma(\bm {x}))
        \le C_\sigma N\epsilon^{1/2}_M
    \end{aligned}
\end{equation*}
since $\min(\sigma(\bm {x}),\tilde\sigma(\bm {x}))\ge 0$.

\textbf{Case 2:} if $\max(\sigma(\bm {x}),\tilde\sigma(\bm {x}))>C_\sigma N\epsilon^{1/2}_M$, then:
\begin{equation*}
    \begin{aligned}
        |\sigma(\bm {x})-\tilde\sigma(\bm {x})|
        =\frac{|\sigma^2(\bm {x})-\tilde\sigma^2(\bm {x})|}{\sigma(\bm {x})+\tilde\sigma(\bm {x})}
        \le \frac{|\sigma^2(\bm {x})-\tilde\sigma^2(\bm {x})|}{\max(\sigma(\bm {x}),\tilde\sigma(\bm {x}))}
        \le C_\sigma N\epsilon^{1/2}_M
    \end{aligned}
\end{equation*}
Therefore, we conclude that:
\begin{equation*}
    \begin{aligned}
        |\sigma(\bm {x})-\tilde\sigma(\bm {x})|\le C_\sigma N\epsilon^{1/2}_M
    \end{aligned}
\end{equation*}
\end{proof}

\begin{proof}[Proof of Theorem 4.4]
    Except for the fact that $\bm{x_t}$ is chosen by minimizing the approximated LCB instead of the true LCB, all conditions of this theorem are the same as in Lemma 5.8 of \citet{Srinivas2012GPregret}. Therefore, the conclusions of Lemma 5.5 and Lemma 5.7 in \citet{Srinivas2012GPregret} still hold here, and we first restate them without repeating proofs for simplicity:
    \begin{equation*}
        \begin{aligned}
            |f(\bm {x_t}) - \mu_{t-1}(\bm{x_t})|&\le \beta_t^{1/2}\sigma_{t-1}(\bm{x_t})\\
            |f(\bm{x^*})-\mu_{t-1}(\bm{[x^*]_t})|&\le\beta_{t-1}(\bm{[x^*]_t})+1/t^2
        \end{aligned}
    \end{equation*}
    where $\bm{[x^*]_t}$ is the closest point in $D_t$ to $\bm{x^*}$, and $D_t\subset D$ is a discretization at $t$-th iteration.

    Recall the definition of $\bm{x_t}$, we have:
    \begin{equation*}
        \begin{aligned}
            \tilde \mu_{t-1}(\bm {x_t}) - \beta_t^{1/2}\tilde\sigma_{t-1}(\bm{x_t})\le \tilde \mu_{t-1}(\bm{[x^*]_t}) - \beta_t^{1/2}\tilde\sigma_{t-1}(\bm{[x^*]_t})
        \end{aligned}
    \end{equation*}
    Combining those three inequalities with our inequalities from Theorem 4.3 gives
    \begin{equation*}
        \begin{aligned}
            r_t&=f(\bm{x_t})-f(\bm{x^*})\\
            &\le \mu_{t-1}(\bm{x_t})+\beta_t^{1/2}\sigma_{t-1}(\bm{x_t})-\mu_{t-1}(\bm{[x^*]_t})+\beta_t^{1/2}\sigma_{t-1}(\bm{[x^*]_t})+1/t^2\\
            &\le \tilde\mu_{t-1}(\bm{x_t})+\beta_t^{1/2}\tilde\sigma_{t-1}(\bm{x_t})-\tilde\mu_{t-1}(\bm{[x^*]_t})+\beta_t^{1/2}\tilde\sigma_{t-1}(\bm{[x^*]_t})+1/t^2+2C_{\mu}t^2\epsilon_{M_t}+2C_\sigma \beta_t^{1/2}t\epsilon^{1/2}_{M_t}\\
            &\le 2\beta_t^{1/2}\tilde\sigma_{t-1}(\bm{x_t})+1/t^2+2C_{\mu}t^2\epsilon_{M_t}+2C_\sigma \beta_t^{1/2}t\epsilon^{1/2}_{M_t}\\
            &\le 2\beta_t^{1/2}\sigma_{t-1}(\bm{x_t})+1/t^2+2C_{\mu}t^2\epsilon_{M_t}+4C_\sigma \beta_t^{1/2}t\epsilon^{1/2}_{M_t}
        \end{aligned}
    \end{equation*}
\end{proof}

\section{Experimental Implementation Details}\label{app:exp_details}

\subsection{Problem setup}
For our experiments, the GP models, including the additive GPs, are implemented with a Mat{\'e}rn 3/2 kernel using the GPflow package \citep{Matthews2017GPflow}. In our implementation, kernel variance $\sigma_f^2$ is bounded to be within $[0.05, 20]$ and kernel lengthscale $l$ is bounded to be within $[0.005, 20]$. The GP parameters are optimized over 10 multiple starts, with random initial values for kernel parameters. The set of parameters with minimal negative log likelihood is then selected. We take the noise term $\sigma_{\epsilon}^2 = 10^{-6}$.

For PK-MIQP, all MIQPs are solved using Gurobi v11.0.0 \citep{gurobi2024}. We list relevant solver hyperparameters in Table \ref{tab:Gurobi} and use default values for other hyperparameters. 

\begin{table}[ht!]
    \centering
     \caption{Gurobi hyperparameters.}
    \label{tab:Gurobi}
    \begin{tabular}{cc}
        \toprule
        Parameter & Value \\
        \midrule
        \href{https://www.gurobi.com/documentation/current/refman/nonconvex.html}{\texttt{NonConvex}} & 2 \\
        \href{https://www.gurobi.com/documentation/current/refman/objscale.html}{\texttt{ObjScale}} & 0.5 \\
        \href{https://www.gurobi.com/documentation/current/refman/scaleflag.html}{\texttt{ScaleFlag}} & 1 \\
        \href{https://www.gurobi.com/documentation/current/refman/mipgap2.html}{\texttt{MIPgap}} & 0.5\\
        \href{https://www.gurobi.com/documentation/current/refman/poolsolutions.html}{\texttt{PoolSolutions}} & 10 \\
        \href{https://www.gurobi.com/documentation/current/refman/timelimit.html}{\texttt{TimeLimit}} & 5400\\
        \bottomrule
    \end{tabular}
\end{table}

As mentioned in the ``Initialization Heuristic'' paragraph of Section 4.2, PK-MIQP first solves a sub-problem to initialize the full-problem. The sub-problem drops all variance-related terms, resulting in
\begin{equation*}\label{eq:sub_model}
    \begin{aligned}
        \min\limits_{x\in\mathcal X}~&
        \tilde{\mu}\\
        s.t.~& \tilde{\mu} = \tilde{K}_{xX}\tilde{K}_{XX}^{-1}\bm{y}\\
        &r_i^2=\frac{\|\bm{x}-\bm{x_i}\|_2^2}{l^2},~\forall 1\le i\le N\\
        &\tilde K_{xX_i}=\tilde k(r_i),~\forall 1\le i\le N
    \end{aligned}
\end{equation*}
For sub-problem, we set \texttt{TimeLimit} to $1800$(s); most are solved relatively quickly.

We use scipy \citep{Scipy2020sciPy} to implement state-of-the-art optimizers for comparison, including:

\textbf{L-BFGS-B:} Limited-memory Broyden–Fletcher–Goldfarb–Shanno Bound (L-BFGS-B) \citep{Zhu1997LBFGS} algorithm is a gradient-based solver for minimizing differentiable function. L-BFGS-B uses the first derivative as well as an estimate of the inverse of Hessian matrix to steer the search. It doesn't support constrained minimization problems.

\textbf{COBYLA:} Constrained Optimization BY Linear Approximation (COBYLA) \citep{Powell1994cobyla} algorithm is a derivative-free solver for minimizing scalar functions. COBYLA updates a linear approximation of both the objective and constraints and performs a simplex method within the trust region to solve the problem. It supports constrained minimization problems.

\textbf{trust-constr:} Trust-region algorithm for constrained optimization (trust-constr) \citep{Byrd1999trustconstr} is a gradient-based solver that applies the interior point algorithm to minimize a given function. Both gradients and Hessians are approximated during the solving. It supports constrained minimization problems. 

\textbf{Nelder-Mead:} Nelder-Mead algorithm \citep{Nelder1965neldermead} is a derivative-free method for solving multidimensional optimization problems. It uses a simplex algorithm based on function comparison and doesn't support constrained minimization problems.

\textbf{SLSQP:} Sequential Least SQuares Programming (SLSQP) \citep{Kraft1988slsqp} algorithm is a gradient-based optimizer. It uses the Han-Powell quasi-Newton method combined with BFGS to solve a Lagrange function at each iteration. It supports constrained minimization problems. 

For these methods, the tolerance of termination (\texttt{tol}) is set to $10^{-6}$, and the maximum number of iterations (\texttt{maxiter}) is set to $1000$. The starting point (\texttt{x0}) for applicable methods is set to $\bm{0}$.

\subsection{Benchmarks}
In this section, we provide the formulations for all benchmark functions used in our experiments.

\textbf{Bumpy (1D):}  
\begin{equation}\tag{Bumpy}
    \begin{aligned}
        \min~& -\sum_{i=1}^{6} i \cdot \sin{((i+1) x + i)}\\
    s.t.~&x\in[-10,10]
    \end{aligned}
\end{equation}
\textbf{Multimodal (1D):}
\begin{equation}\tag{Multimodal}
    \begin{aligned}
        \min~& \sin{x} + \sin{\left(\frac{10}{3} x\right)}\\
    s.t.~&x\in[-2.7, 7.5]
    \end{aligned}
\end{equation}
\textbf{Ackley (2D)}
\begin{equation}\tag{Ackley}
    \begin{aligned}
        \min~& -20 \exp{\left(-0.2\sqrt{0.5(x_1^2+x_2^2)}\right)}-\exp{(0.5 (\cos{(2\pi x_1))}+\cos{(2\pi x_2)})}+20+\exp{(1)}\\
        s.t.~& x_1, x_2 \in [-32, 16]
    \end{aligned}
\end{equation}
\textbf{Branin (2D)}
\begin{equation}\tag{Branin}
    \begin{aligned}
        \min~& a(x_2-bx_1^2+cx_1-r)^2+s(1-t)\cos{(x_1)}+s\\
        s.t.~& x_1 \in [-5, 10], x_2 \in [0, 15]
    \end{aligned}
\end{equation}
where $a=1,b=5.1/(4\pi^2),c=5/\pi,r=6,s=10,t=1/(8\pi)$.

\textbf{Rosenbrock (2D)}
\begin{equation}\tag{Rosenbrock}
    \begin{aligned}
        \min~& (1-x_1)^2+100(x_2-x_1^2)^2\\
        s.t.~& x_1 \in [-2, 2], x_2 \in [-1, 3]
    \end{aligned}
\end{equation}

\textbf{Hartmann (3D)}
\begin{equation}\tag{Hartmann}
    \begin{aligned}
        \min~& -\sum_{i = 1}^4 \alpha_i \exp{\left(- \sum_{j = 1}^3 A_{ij} (x_j - P_{ij})^2\right)}\\
        s.t.~& x_1, x_2, x_3 \in [0, 1]
    \end{aligned}
\end{equation}
where $\alpha=(1, 1.2, 3, 3.2)^T$ and
\begin{equation*}
    \begin{aligned}
        A = 
            \begin{pmatrix}
                3 & 10 & 30\\
                0.1 & 10 & 35\\
                3 & 10 & 30 \\
                0.1 & 10 & 35
            \end{pmatrix},~
        P = 10^{-4}
            \begin{pmatrix}
                3689 & 1170 & 2673\\
                4699 & 4387 & 7470\\
                1091 & 8732 & 5547\\
                381 & 5743 & 8828
            \end{pmatrix}
    \end{aligned}
\end{equation*}

\textbf{Michalewicz (5D)}
\begin{equation}\tag{Michalewicz}
    \begin{aligned}
        \min~& -\sum_{i=1}^5\sin{(x_i)}\sin^{20}{\left(\frac{i x_i^2}{\pi}\right)}\\
        s.t.~& x_i \in [0, \pi], i=1, \cdots, 5
    \end{aligned}
\end{equation}

\textbf{Constrained KS224 (2D)} For thise case, we solve the following constrained minimization problem with the 2D KS224 function \citep{ks2024ks224} as the objective.

\begin{equation}\tag{Constrained KS224}
    \begin{aligned}
        \min~& 2 x_1^2 + x_2^2 - 48 x_1 - 40 x_2\\
        s.t.~& -(x_1 + 3 x_2) \leq 0\\
         &-(18 - x_1 - 3 x_2) \leq 0\\
         &-(x_1 + x_2) \leq 0\\
         &-(8 - x_1 - x_2) \leq 0\\
         &~x_1, x_2 \in [0, 6]
    \end{aligned}
\end{equation}

\textbf{SVM hyperparameter tuning:} The 20 newsgroup dataset \citep{Lang1995news} is a classic dataset consisting of posts on 20 topics for text classification. We implement a simple pipeline for this text classification in scikit-learn \citep{Sklearn2011scikitlearn} comprising first using a TF-IDF vectorizer to convert text to vectors and then applying C-Support Vector Classification (SVC). We use the default settings for hyperparameters in SVC except leaving the regularization hyperparameter $C \in [0.01, 1000]$ and kernel coefficient $\gamma \in [0.01, 1000]$ to be tuned. Negative 5-fold cross validation accuracy is set as the objective to be minimized.

\section{Numerical results using RBF kernel}\label{app:rbf}
We perform the same single acquisition function optimization experiment as in Section \ref{subsec:BO_one_step}, but using the RBF kernel instead of the Matern 3/2 kernel. Table \ref{tab:one_iter_rbf} reports the mean of the optimal LCB values found by different optimizers, along with half of the standard deviation. For Bayesian optimization, we empirically demonstrate the performance of PK-MIQP with RBF kernel in two benchmarks as shown in Figure \ref{fig:rbf}.

\begin{table*}[h]
    \centering
    \caption{Comparison of solvers on optimizing random acquisition functions using RBF kernel. The mean of the optimal LCB values found over the 20 replications is reported with 0.5 standard deviation in parentheses. PK-MIQP consistently outperforms other gradient- and sampling-based methods.}
    \label{tab:one_iter_rbf}
    \begin{tabular}{cccccc}
    \toprule
    method & 1D & 2D & 3D & 4D & 5D \\
    \midrule
    L-BFGS-B & -0.67(0.68) & -0.90(0.42) & -2.03(1.85) & -2.77(1.45) & -2.26(1.31)\\
    Nelder-Mead & -0.98(0.63) & -0.87(0.67) & -1.53(0.74) & -2.70(1.43) & -2.12(1.38)\\
    COBYLA & -1.25(0.82) & -1.44(0.60) & -2.18(1.83) & -1.47(0.99) & -2.49(1.37)\\
    SLSQP & -1.04(1.00) & -0.72(0.41) & -2.34(1.83) & -2.53(1.49) & -2.28(1.29)\\
    trust-constr & -0.57(0.68) & -1.27(0.38) & -1.90(0.70) & -3.06(1.56) & -2.89(1.32)\\
    PK-MIQP & \textbf{-1.82(0.71)} & \textbf{-1.99(0.49)} & \textbf{-2.94(0.65)} & \textbf{-4.20(1.87)} & \textbf{-3.10(1.28)}\\
    \bottomrule
    \end{tabular}
\end{table*}

\begin{figure*}[h]
    \centering
    \begin{subfigure}[b]{0.45\textwidth}
         \centering
         \includegraphics[width=\textwidth]{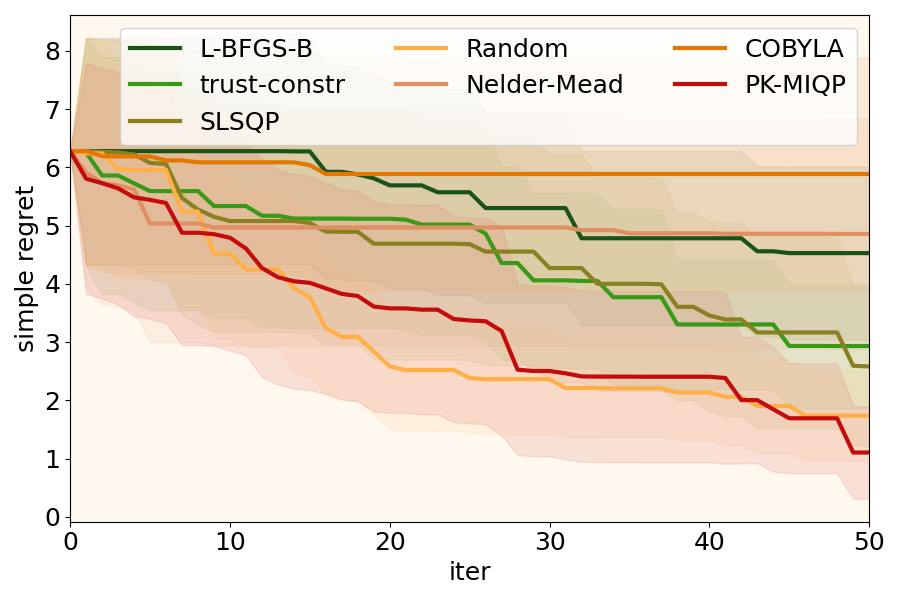}
         \caption{1D Bumpy}
         \label{fig:bumpy}
     \end{subfigure}
     \begin{subfigure}[b]{0.45\textwidth}
         \centering
         \includegraphics[width=\textwidth]{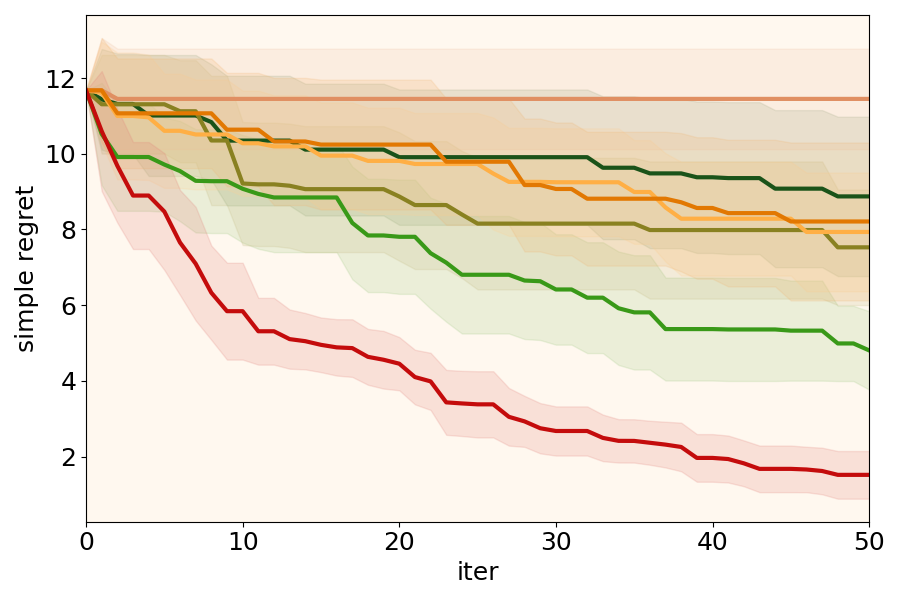}
         \caption{2D Ackley}
         \label{fig:ackley}
     \end{subfigure}
    \caption{Numerical results on Bayesian optimization using PK-MIQP with RBF kernel and the state-of-the-art minimizers. The mean with $0.5$ standard deviation of simple regret is reported over $20$ replications. PK-MIQP is similarly applicable to the RBF kernel and again outperforms other minimizers.}
    \label{fig:rbf}
\end{figure*}

\end{document}